\documentclass{article}%
\usepackage{graphicx}
\usepackage{amsmath}
\usepackage{amsfonts}
\usepackage{amssymb}%
\providecommand{\U}[1]{\protect \rule{.1in}{.1in}}
\newtheorem{theorem}{Theorem}

\newtheorem{corollary}[theorem]{Corollary}

\newtheorem{definition}[theorem]{Definition}

\newtheorem{lemma}[theorem]{Lemma}

\newtheorem{proposition}[theorem]{Proposition}
\newtheorem{remark}[theorem]{Remark}

\newenvironment{proof}[1][Proof]{\textbf{#1.} }{\  \rule{0.5em}{0.5em}}
\begin{document}

\title{\textbf{Morse foliations of codimension one on the sphere }$S^{3}$}
\author{Charalampos Charitos}
\maketitle

\begin{abstract}
Morse foliations of codimension one on the sphere $S^{3}$ are studied and the
existence of special components for these foliations is derived. As a
corollary the instability of Morse foliations of $S^{3}$ can be proven
provided that all the leaves are not simply connected.

\end{abstract}

\section{\bigskip Introduction}

In this paper we study Morse foliations $\mathcal{F}$ of codimension one and
of class $C^{\infty}$ of the sphere $S^{3}.$ Such a foliation $\mathcal{F}$
has finitely many singularities of conic or central type; outside these
singularities $\mathcal{F}$ is a foliation in the usual sense. For each
singularity $s$ there is an open neighborhood $U$ of $s$ homeomorphic to
$R^{3},$ such that the leaves of the induced foliation $\mathcal{F}_{|U}$ are
defined as the level sets of a Morse function $f:U\rightarrow R.$ Such a
neighborhood $U$ will be called a \textit{trivial neighborhood} of $s.$ The
complete definition of Morse foliations will be given at the beginning of the
next section.

Each Morse foliation can be transformed by a modification which will be
referred to as a \textit{Morse modification. }This modification is defined
locally and moves conic singularities to nearby leaves. A \textit{\ Morse
component }$\mathcal{M}$ is a Morse foliation defined on the solid torus
$D^{2}\times S^{1}$ which contains one center, one conic singularity and all
the non-singular leaves of $\mathcal{M}$ are either tori or 2-spheres. \ 

The main result of this work is the following.

\begin{theorem}
\noindent \  \textbf{\ }\label{main theorem} \textit{Let }$\mathcal{F}%
$\textit{\ be a Morse foliation of codimension one of }$S^{3}$ which admits at
least one non-simply connected leaf. \textit{\ Then, modulo Morse
modifications, }$\mathcal{F}$\textit{\ contains either a Reeb component or a
Morse component. }
\end{theorem}

As a corollary of this theorem we may prove the following theorem which
answers affirmatively a conjecture of Rosenberg and Roussarie \cite{[Ros-Rou]}.

\begin{theorem}
\label{stability} The only stable Haefliger structures of $S^{3}$ are those
defined by a Morse function having distinct critical values and whose level
surfaces are all simply connected.
\end{theorem}

Generally, Morse foliations are an important generalization of foliations and
their study has a long history, see for example \cite{[Reeb1]}, \cite{[Ferry]}%
, \cite{[Wagneur]}, \cite{[Cam Sca]}, \cite{[Cam Sca1]}, \cite{[Rosati]}. In
particular, in \cite{[Rosati]} classical compact leaf theorems of Haefliger
and Novikov are generalized (see for instance, Theorem 9.1). Nevertheless, the
latter are proved under conditions on the set of singular points of a singular
foliation which are strong and not related to our setup.

Despite the long time period, the theory of singular foliations on 3-manifolds
did no grow as much as that of regular foliations. The reason is that each
closed 3-manifold $M$ can be foliated by a regular foliation of codimension
one. However, if $M$ is compact with $\partial M\neq \varnothing$ and if the
genus of $\partial M$ is greater that one, then in order to foliate $M$
tagentially to the boundary we need to consider Morse foliations. On the other
hand, Haefliger structures are an even greater generalization of foliations
\cite{[Haefliger]}, \cite{[Lawson]} and they are quite useful to make a
homotopy-theoretic approach to foliations. Below, we will see that Morse
foliations are generic in the space of Haefliger structures of codimension one
fact that makes the class of Morse foliations particular
interesting.\smallskip

The present work is organized as follows:

In section 2 definitions are given and basic notions are explained in order to
fix the terminology. Typical examples of Morse foliations of $S^{3}$ are also constructed.

In section 3 we recall the definition of vanishing cycle and we generalize
Novikov's theorem for Morse foliations. We also introduce the notion of
anti-vanishing cycle and examine its consequences.

In section 4 and 5 we prove Theorem \ref{main theorem}. Actually, we prove
that if $\mathcal{F}$ does not contain a Reeb or a Morse component and if
$\mathcal{F}$ does not have trivial pair of singularities then $\mathcal{F}$
contains either a \textit{truncated Reeb component} or a \textit{truncated
Morse component }or $\mathcal{F}$ has two centers and all the other leaves of
it are 2-spheres. Notice that, a truncated Reeb or Morse component can be
transformed to a Reeb or Morse component respectively by a Morse modification.

In section 6 we prove Theorem \ref{stability}. \begin{figure}[ptb]
\begin{center}
\includegraphics[scale=0.76]
{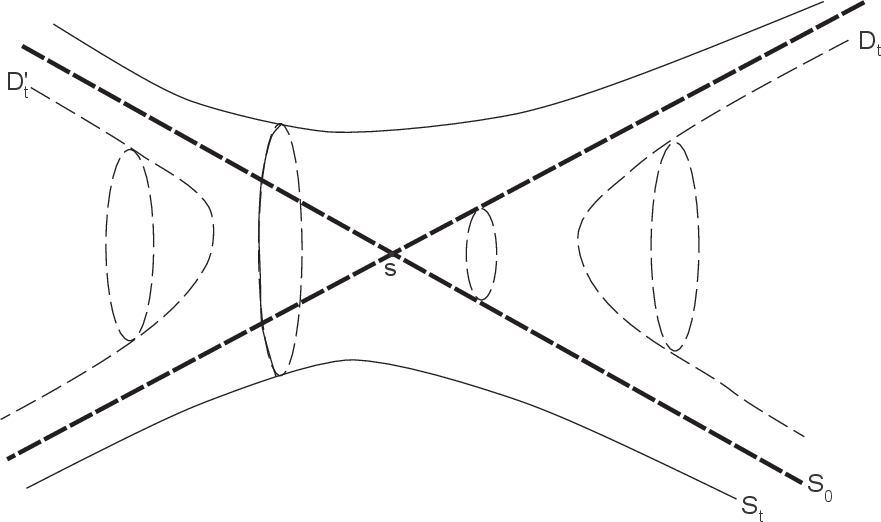}
\end{center}
\caption{The level surfaces in a trivial neighborhood of a conic singularity.
}%
\label{sxima1}%
\end{figure}

\section{Definitions and Preliminaries}

\subsection{Morse foliations and Morse modification}

First we will give the precise definition of Morse foliations of codimension
one of a compact $n$-manifold.

\begin{definition}
Let $M$ be a compact, orientable $n$-manifold of class $C^{r},$ $n\geq2$ and
$r\geq2.$ By a Morse foliation $\mathcal{F}$ of codimension one of $M$ which
is tangent to the boundary $\partial M$ if $\partial M\neq \varnothing,$ we
mean a decomposition of $M$ into a union of disjoint connected subsets
$\{L_{a}\}_{a\in \mathcal{A}},$ called the leaves of the foliation such that:

(1) There are finitely many points $s_{i}$ in $M$ which have the following
property: For each $s_{i}$ there exists an open neighborhood $U_{i}$ of
$s_{i}$ in $M$ such that for each leaf $L_{a},$ the components of $U_{i}\cap
L_{a}$ are described as level sets of a Morse function $f_{i}:U_{i}%
\rightarrow(-1,1)$ of class $C^{r}.$ Furthermore, the point $s_{i}$ is the
unique non-degenerate critical point of $f_{i}.$

(2) For each point $p\in M,$ $p\neq s_{i}$ there exists an open neighborhood
$U$ of $p$ such that for each leaf $L_{a},$ the components of $U\cap L_{a}$
are described as level sets of a submersion $f_{i}:U\rightarrow(-1,1)$ of
class $C^{r}.$

The points $s_{i}$ are called singularities of $\mathcal{F}.$ If the set
$\{s_{i}\}$ is empty then $\mathcal{F}$ is a foliation in the usual sense and
will referred to as a regular foliation or simply as foliation of $M.$ The
open neighborhood $U_{i}$ of $s_{i}$ will be called trivial, while the
neighborhood $U$ of $p\neq s_{i}$ is called a local chart of $\mathcal{F}$.

In particular, for each $s_{i}$ there is a system of coordinates
$x=(x_{1},\cdots,x_{n})$ on a neighborhood $U_{i}$ of $s_{i}$ such that
\[
f_{i}(x)=-\sum_{i=1}^{k}x_{i}^{2}+\sum_{i=k+1}^{n}x_{i}^{n}.
\]

In each $U_{i}$ with respect to the system of coordinates $(x_{1},\cdots
,x_{n}),$ the singularity $s_{i}$ is identified with the point $(0,\cdots,0),$
i.e. $f_{i}^{-1}(0)=s_{i}.$

In both special cases $f_{i}(x)=-\sum_{i=1}^{n}x_{i}^{2}$ and $f_{i}%
(x)=\sum_{i=1}^{n}x_{i}^{2},$ the singularity $s_{i}$ is called a center. For
the centers, the level surfaces are defined for each value in $(-1,0]$ or in
$[0,1).$

If a singularity is not a center it is called conic.

Existence of conic singularities on the boundary $\partial M$ of $M$ is
analogously defined and are allowed.
\end{definition}

In the following we consider Morse foliations $\mathcal{F}$ of codimension one
and of class $C^{\infty}$ of the sphere $S^{3}.$ If the foliated 3-manifolds
are not homeomorphic to $S^{3}$ we will always mention it.

Each leaf of $\mathcal{F}$ is one of the following three types:

\begin{itemize}
\item a regular leaf which is an orientable surface;

\item a leaf with finitely many \textit{conic singularities }$s_{i};$

\item a \textit{center.}
\end{itemize}

A leaf containing a conic singularity will be called \textit{singular leaf}.

We always assume that $\mathcal{F}$ is orientable which means that, if we
remove from $S^{3}$ all the singularities of $\mathcal{F}$ then we take a
3-manifold $N$ such that $\mathcal{F}_{|N}$ is an orientable foliation of
codimension one in the usual sense. Notice finally that the space of Morse
foliations of $S^{3}$ is equipped with the $C^{1}$-topology.

Regular foliations of the sphere $S^{3}$ are extensively studied and they are
dominated by Novikov's theorem \cite{[Novikov]} which asserts that each
orientable foliation of codimension one of $S^{3}$ has a Reeb component and
thus non-compact leaves homeomorphic to $R^{2}.$ For Morse foliations the
situation is (at least seemingly) completely different. We will construct
below Morse foliations whose all leaves are compact, simply connected or not,
as well as, foliations without any compact leaf. However, we will see that,
except the case where all the leaves are simply connected, two kinds of
components appear. These components will be referred to as \textit{truncated
Reeb} or \textit{Morse} \textit{components}.

Let $s$ be a conic singularity and we will give below the definition of
\textit{Morse modification} around $s.$ Notice that this modification is local
in the sense that it takes place in a trivial neighborhood $U$ of $s.$

Consider a neighborhood $U$ of $s$ homeomorphic to $R^{3}$ such that the
leaves of the induced foliation $\mathcal{F}_{|U}$ are the level sets of a
Morse function $f:U\rightarrow R$ of index 1 or 2. Let $S_{t}=f^{-1}(t).$
Therefore, we have locally three types of level surfaces:

\noindent(1) $S_{t}$ is a disjoint union of two open discs $D_{t},$
$D_{t}^{\prime},$ for each $t\in(-\infty,0);$

\noindent(2) $S_{t}$ is a cylindrical leaf homeomorphic to $S^{1}\times(0,1),$
for each $t\in(0,\infty),$

\noindent(3) $S_{0}$ contains $s$ and will be referred to as a \textit{double
cone containing }$s,$ see Figure 1.

\begin{definition}
Let $U$ be a neighborhood of $s$, as above. By a $\emph{Morse}$
$\emph{modification}$\textit{\ of }$\mathcal{F}$\textit{\ around }$s$ we mean
a local transformation which produces a new Morse foliation on $S^{3},$ such
that one of the following two modifications $(A)$ or $(B)$ is performed.

$(A)$ For some $t_{0}\in(-\infty,0)$ we consider points $p_{t_{0}}\subset$
$D_{t_{0}}$ and $p_{t_{0}}^{\prime}\subset D_{t_{0}}^{\prime}$ and we glue
$D_{t_{0}}$ and $D_{t_{0}}^{\prime}$ by identifying $p_{t_{0}}$ with
$p_{t_{0}}^{\prime}.$ Then, a double cone $S_{t_{0}}$ is constructed
containing as conic singularity the point $p_{t_{0}}\equiv p_{t_{0}}^{\prime
}.$ Each $S_{t},$ $t\in(-\infty,$ $t_{0})$ is a disjoint union of two open
discs and each leaf $S_{t},$ $t\in(t_{0},$ $\infty)$ is cylindrical.

$(B)$ For some $t_{0}\in(0,\infty)$ we consider an essential simple curve
$c_{t_{0}}\subset S_{t_{0}}.$ By shrinking $c_{t_{0}}$ to a point $s_{t_{0}},$
the leaf $S_{t_{0}}$ is transformed to a double cone containing the point
$s_{t_{0}}$ as a conic singularity. Each $S_{t},$ $t\in(-\infty,$ $t_{0})$ is
a disjoint union of two open discs and each leaf $S_{t},$ $t\in(t_{0},$
$\infty)$ is cylindrical.

Obviously a Morse modification can be defined in any 3-manifold $M$ equipped
with a Morse foliation $\mathcal{F}$ and the modification $(A)$ is the inverse
of $(B)$ and vice versa. \medskip
\end{definition}

\begin{figure}[ptb]
\begin{center}
\includegraphics[scale=0.87]
{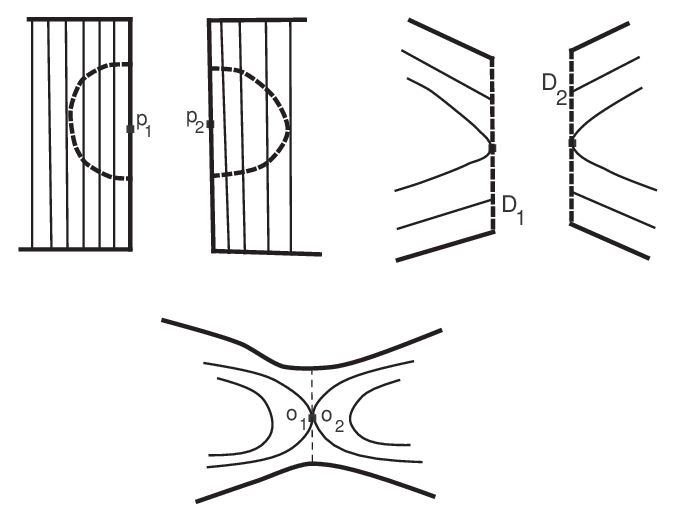}
\end{center}
\caption{The connected sum of foliations.}%
\label{sxima2}%
\end{figure}

\begin{remark}
\label{Morse modification}Obviously, modulo Morse modifications, we may always
assume that a leaf of $\mathcal{F}$ does not contain more than one conic
singularity. More precisely, the following property holds:

\begin{itemize}
\item \noindent For every Morse foliation $\mathcal{F}$ and every neighborhood
$\mathcal{U}$ of $\mathcal{F}$, with respect to the $C^{1}$- topology, there
is a Morse foliation $\mathcal{F}^{\prime}$ in $\mathcal{U}$ such that every
leaf of $\mathcal{F}^{\prime}$ does not contain more than one conic singularity.
\end{itemize}

In what follows and in order to avoid technical difficulties, we always assume
that each leaf of a Morse foliation $\mathcal{F}$ contains at most one conic
singularity. It is important to notice here that this assumption does not
affect the dynamic behavior of $\mathcal{F}$.\bigskip
\end{remark}

\subsection{The connected sum of foliations (see \cite{[Ros-Rou]})\medskip}

Let $D$ be an embedded 2-disc in $S^{3}$ which does not contain any
singularity of $\mathcal{F}.$ The definition of being $D$ in \textit{general
position} with respect to $\mathcal{F}$ is standard for regular foliations of
codimension one (see for example \cite{[CaCo]}, Ch. 7, Def. 7.1.5) and it is
readily generalized for Morse foliations.\smallskip

Now, the connected sum of foliated 3-manifolds will be defined.

Let $(M_{1},\mathcal{F}_{1})$ and $(M_{2},\mathcal{F}_{2})$ be two compact
foliated 3-manifolds with non-empty boundary. Here we assume that each
$\mathcal{F}_{i}$ is either a Morse foliation or a regular 2-dimensional
foliation tangent to $\partial M_{i}.$ Let $p_{i}\in \partial M_{i}$ and let
$B_{i}$ be a \textit{\ }neighborhood of $p_{i}$ in $M_{i}$ such that:

(1) $B_{i}$ is homeomorphic to a closed 3-ball and there do not exist conic
singularities or centers in $B_{i};$

(2) $\partial B_{i}$ is the union of two 2-discs $D_{i}$ and $E_{i}$ such
that: $D_{i}\cap E_{i}=\partial D_{i}=\partial E_{i}$ and $E_{i}%
\subset \partial M_{i};$

(3) $D_{i}$ is in general position with respect to $\mathcal{F}_{i}$ and
$\mathcal{F}_{i|D_{i}}$ is a foliation by concentric circles of center
$o_{i}.$

The induced foliation $\mathcal{F}_{i|D_{i}}$ by concentric circles will be
referred to as a \textit{trivial foliation} on the disc $D_{i}.$

From Reeb stability theorem \cite{[Reeb]} we deduce that $\mathcal{F}_{i}$
induces in $B_{i}-\{o_{i}\}$ a product foliation by 2-discs. The neighborhood
$B_{i}$ will be referred to as a \textit{trivially foliated neighborhood} of
$p_{i}.$ If $N_{i}=(M_{i}-B_{i})\cup D_{i},$ we may form the connected sum
$M_{1}\#M_{2}$ by identifying $N_{1}$ with $N_{2},$ via a diffeomorphism
$f:D_{1}\rightarrow D_{2}$ which sends $o_{1}$ to $o_{2}$ and the leaves of
$\mathcal{F}_{1|D_{1}}$ onto the leaves of $\mathcal{F}_{2|D_{2}}.$ In this
way, a Morse foliation $\mathcal{F}_{1}\# \mathcal{F}_{2}$ is defined on
$M_{1}\#M_{2};$ this foliation has one more conic singularity at the point
$o_{1}\equiv o_{2}$ and is tangent to $\partial(M_{1}\#M_{2}).$ Henceforth,
the connected sum of $(M_{1},\mathcal{F}_{1})$ and $(M_{2},\mathcal{F}_{2})$
will be denoted by $(M_{1}\#M_{2},\mathcal{F}_{1}\# \mathcal{F}_{2}).$
Obviously $\partial(M_{1}\#M_{2})=\partial M_{1}\# \partial M_{2}$, where
$\partial M_{1}\# \partial M_{2}$ denotes the connected sum of surfaces
$\partial M_{1}$ and $\partial M_{2}.$ In Figure 2 and in in a 2-dimensional
slice, the steps of the operation of connected sum are depicted. This
operation is defined in \cite{[Ros-Rou]} by means of equations.

\begin{figure}[ptb]
\begin{center}
\includegraphics[scale=0.66]
{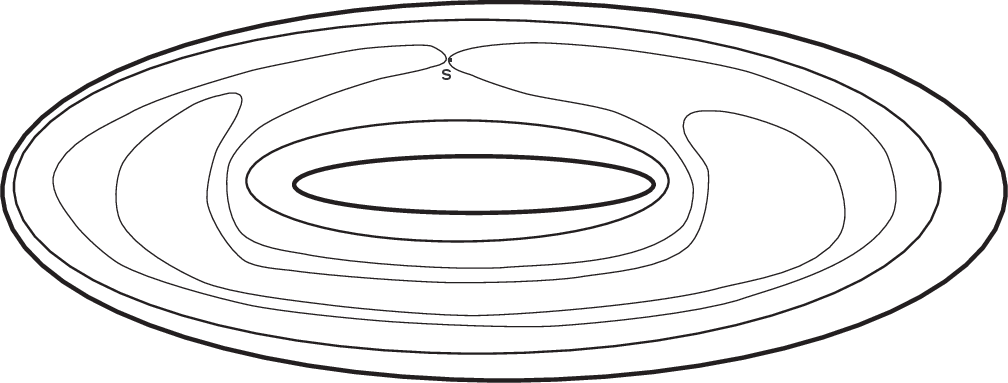}
\end{center}
\caption{A Morse component.}%
\label{sxima3}%
\end{figure}\bigskip

\subsection{Examples of Morse foliations}

\medskip

Below we will give examples of Morse foliations of compact
3-manifolds.\smallskip

(1) We may foliate the 3-dimensional ball $B^{3}$ by concentric spheres. Let
denote this foliation by $\mathcal{S}.$ This is a Morse foliation with one
center. By gluing two such 3-balls along their boundaries, we take a Morse
foliation $\mathcal{F}$ on $S^{3}$ which has two centers and all others leaves
are 2-spheres.

Similarly to the previous example, Morse foliations of $S^{3}$ can be
constructed whose all leaves are compact and simply connected. For, we form
the connected sum of finitely many $(B_{i},\mathcal{S}_{i}),$ where each
$B_{i}$ is a 3-ball and each $\mathcal{S}_{i}$ is a Morse foliation by
concentric spheres. Let denote by $(\overline{B},\mathcal{S})$ the resulting
foliated manifold. Obviously $\overline{B}$ is a 3-ball and each leaf of
$\mathcal{S}$ is compact and simply connected. Gluing two such copies
$(\overline{B}_{1},\mathcal{S}_{1}),$ $(\overline{B}_{2},\mathcal{S}_{2})$, by
identifying the boundaries $\partial \overline{B}_{1}$ and $\partial
\overline{B}_{2}$ via a diffeomorphism, we take a foliation $\mathcal{C}$ on
$S^{3}$ whose all leaves are compact and simply connected. The foliation
$\mathcal{C}$ will be referred to as a \textit{Morse foliation of }$S^{3}%
$\textit{\ with simply connected leaves}. $\bigskip$

(2) Let $(T,\mathcal{R})$ be a Reeb component, i.e. $T$ is homeomorphic to the
solid torus $S^{1}\times D^{2};$ all leaves in the interior of $T$ are planes
and the boundary $\partial T$ is a toral leaf. Considering two Reeb components
$(T_{1},\mathcal{R}_{1})$ and $(T_{2},\mathcal{R}_{2})$ we form the connected
sum $(T_{1}\#T_{2},\mathcal{R}_{1}\# \mathcal{R}_{2}).$ Obviously
$P=T_{1}\#T_{2}$ is a solid pretzel and $\mathcal{G}=\mathcal{R}_{1}\#
\mathcal{R}_{2}$ is Morse foliation whose all the leaves in the interior of
$P$ are non-compact. Finally, considering two such copies $(P,\mathcal{G})$
and $(P^{\prime},\mathcal{G}^{\prime})$ of foliated pretzels and gluing them
properly along their boundaries we take a Morse foliation $\mathcal{F}$ on
$S^{3}$ with a single compact leaf, say $L_{0},$ which is a surface of genus 2.

Now we may construct a Morse foliation on $S^{3}$ without compact leaves (see
\cite{[Ros-Rou]}, Prop. 3.1). For this, we consider a 2-sphere $S$ embedded in
$S^{3}$ such that:

$a)$ $S=D\cup D^{\prime}$, where $D,D^{\prime}$ are 2-discs and $D\cap
D^{\prime}=c$ is a separating curve in $L_{0};$

$b)$ $D\subset P,$ $D^{\prime}\subset P^{\prime}$ and the induced foliation
$\mathcal{G}_{|D}$ (resp. $\mathcal{G}_{|D^{\prime}}^{\prime})$ is a trivial
foliation by concentric circles.

Let $U$ be a neighborhood of $c$ in $S^{3},$ homeomorphic to $S^{1}%
\times \lbrack-1,1]\times \lbrack-1,1]$ such that $\mathcal{F}_{|U}$ is a
trivial product foliation of the form $S^{1}\times \lbrack-1,1]\times \{t\},$
$t\in \lbrack-1,1]$ with $S^{1}\times \lbrack-1,1]\times \{0\} \subset L_{0}$ and
$S^{1}\times \{0\} \times \lbrack-1,1]\subset S.$ Let denote by $L_{0}^{+}$,
$L_{0}^{-}$ the connected components of $L_{0}-c.$ By a small perturbation of
$\mathcal{F}$ in $U$ we may destroy the compact leaf $L_{0}$ by leading
$L_{0}^{+}$ inside $P^{\prime}$ and $L_{0}^{-}$ inside $P,$ see
\cite{[Ros-Rou]} for a formal description of this perturbation. In this way we
take a foliation without compact leaves.

(3) Consider a Morse foliation $\mathcal{M}$ on the solid torus $T$ such that
$\mathcal{M}$ has one center $c$ and one conic singularity $s.$ The regular
leaves of $\mathcal{M}$ are either spheres or tori parallel to $\partial T.$
Moreover, $\mathcal{M}$ has a singular leaf $C$ containing the conic
singularity $s$ and $C-s$ is homeomorphic to 2-sphere minus two points. The
foliated manifold $(T,\mathcal{M})$ will be called a \textit{Morse component},
see Figure 3. The singular leaf $C$ will be called a \textit{pseudo-torus.
}Let denote by $V$ the component of $T-C$ which is foliated by spherical
leaves. By abusing the language, we will say that the singular leaf $C$ with
all the spherical leaves of $V$ define also a Morse component on the closure
$\overline{V}$ of $V;$ if we want to be more precise this component will be
called a \textit{pseudo-Morse component.}

Now, considering two copies of $(T_{i},\mathcal{M}_{i})$, $i=1,2$ of Morse
components we form the connected sum $(T_{1}\#T_{2},\mathcal{M}_{1}\#
\mathcal{M}_{2}).$ Obviously $T_{1}\#T_{2}$ is a solid pretzel and
$\mathcal{M}_{1}\# \mathcal{M}_{2}$ has two centers and two conic
singularities. Obviously, we may construct a Morse foliation on $S^{3}$
containing Morse components.

\begin{definition}
Let $(T,\mathcal{R})$ (resp. $(T,\mathcal{M}))$ be a Reeb component (resp.
Morse component). We remove the interior of a trivially foliated neighborhood
$B$ of some point $p\in \partial T$ and let $G=T-int(B).$ Then $(G,\mathcal{R}%
_{|G})$ (resp. $(G,\mathcal{M}_{|G}))$ will be called a \emph{truncated Reeb
(resp. truncated Morse)} component.\noindent \smallskip \smallskip
\end{definition}

\subsection{Elimination of trivial pairs of singularities\medskip}

Now, we will describe a method that allow us to modify a Morse foliation
$\mathcal{F}$ by transforming the induced foliation $\mathcal{F}_{|B}$ in
certain 3-balls $B$ which are trivially foliated by concentric spheres. In
this manner the Morse foliation $\mathcal{F}$ is transformed to a Morse
foliation $\mathcal{F}^{\prime}$ which is simpler than $\mathcal{F}$. On the
other hand, proving Theorem \ref{main theorem} for $\mathcal{F}^{\prime}$ we
may see that it remains valid for $\mathcal{F}$.

\begin{figure}[ptb]
\begin{center}
\includegraphics[scale=0.66]
{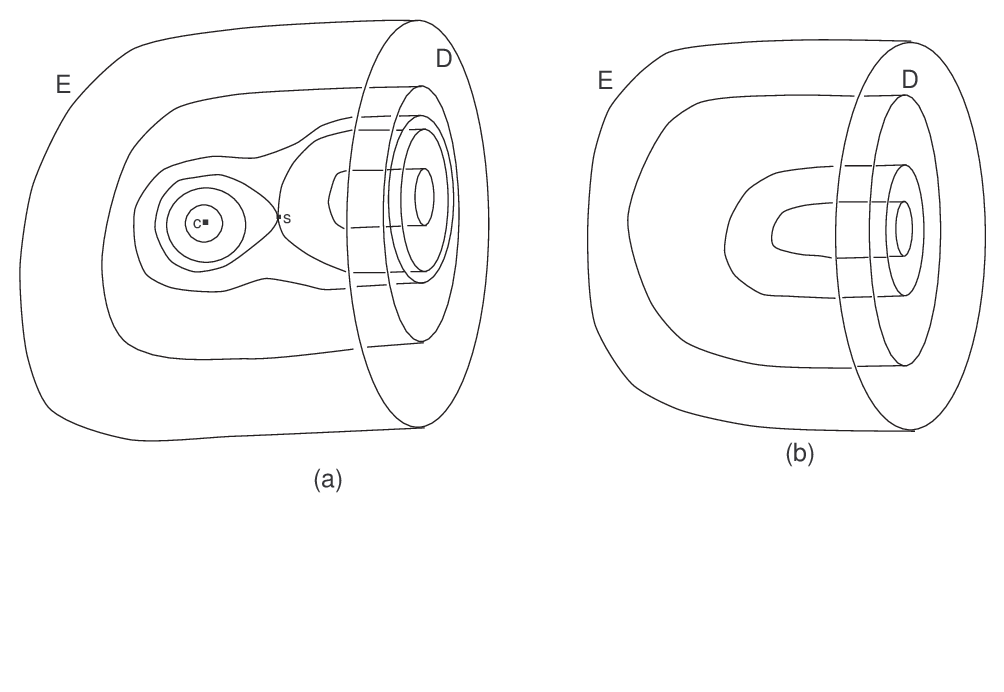}
\end{center}
\caption{Elimination of trivial pair of singularities.}%
\label{sxima4}%
\end{figure}

First we will eliminate some pairs of singularities.

If $c$ is a center of $\mathcal{F}$ then around $c$ there are spherical
leaves. So, we assume that there exists a $C^{\infty}$-map $g:[0,1]\times
S^{2}\rightarrow S^{3}$ such that:

$1)$ $g(\{0\} \times S^{2})=c,$

$2)$ $g:(0,1]\times S^{2}\rightarrow S^{3}$ is an embedding and $g(\{t\}
\times S^{2}),$ $t\in(0,1)$ are parallel spherical leaves,

$3)$ $g(\{1\} \times S^{2})$ contains a conic singularity $s.$

Let $S=g(\{1\} \times S^{2})$ and let $V$ be the component of $S^{3}-S$
containing the center $c.$

\begin{definition}
\label{trivial pair singularities} We will say that the pair $(c,s)$ defines a
\textit{trivial pair of singularities and } $\overline{V}=V\cup S$ will be
called a \textit{trivial bubble}.
\end{definition}

Obviously, we may find discs $D$ and $E$ embedded in $S^{3}$ such that:

$(i)$ $\mathcal{F}_{|D}$ is trivial;

$(ii)$ $E$ is contained in a leaf of $\mathcal{F}$, $E\cap D$ is a leaf of
$\mathcal{F}_{|D}$ and $E\cup D$ bounds a 3-ball $B$ containing $\overline{V};
$

$(iii)$ $\mathcal{F}_{|B}$ consists of a center $c,$ of spherical leaves, of
discs parallel to $E$ and of a singular leaf containing $s,$ see Figure 4(a).

Keeping the discs $D$ and $E$ fixed we may replace the foliation
$\mathcal{F}_{|B}$ by a foliation by parallel discs such that $\mathcal{F}%
_{|D}$ stays invariant, see Figure 4(b). In this way we may remove the
singularities $c$ and $s.$

Notice here that the elimination of a trivial pair of singularities is also
defined in Proposition 6.2 of \cite{[Rosati]}.

Finally, it is not difficult to see that any Morse foliation of $S^{3}$ with
compact simply connected leaves, is led to a foliation with two centers and
spherical leaves by removing successively all trivial pairs of singularities.
In fact, eliminating all these trivial pairs we obtain a foliation with all
its leaves simply connected and without conic singularities. Thus our
assertion follows.

\subsection{Elimination of bubbles}

In this paragraph we assume that $\mathcal{F}$ does not have trivial pair of singularities.

Contrary to the regular foliations, if we allow singularities we may construct
in the interior of a 3-ball $B$ \textit{non-trivial} Morse foliations which
are tangent to the boundary $\partial B$ i.e. $\mathcal{F}_{|B}$ does not
consist of concentric 2-spheres parallel to $\partial B.$

Let $U$ be a local chart of $\mathcal{F}$ i.e. the closure $\overline{U}$ of
$U$ is homeomorphic to $D^{2}\times \lbrack0,1]$ and the leaves of
$\mathcal{F}_{|\overline{U}}$ are of the form $D^{2}\times \{t\},$ $t\in
\lbrack0,1].$ To the foliation $\mathcal{F}$ we may perform the following
modification. In the interior of some $D_{t_{0}}=D^{2}\times \{t_{0}\}$ we pick
a point, say $s,$ and we join to $D_{t_{0}}$ a 2-sphere $S\subset U$ so that
$S\cap D_{t_{0}}=\{s\}.$ Obviously, $S$ defines a 3-ball $B$ in $U.$ In the
interior of $B$ we consider a non-trivial Morse foliation. We also assume that
all the other discs $D_{t}$ with $t\neq t_{0}$ remain invariants under an
isotopy taking place in $U.$ In this way, a new Morse foliation $\mathcal{F}%
^{\prime}$ is constructed which has $s$ as a conic singularity; in
\textbf{\ }Figure 5(a) and 5(b) a 2-dimensional section of the whole process
is drawn. The 3-ball $B$ will be called a (non-trivial) \textit{bubble} of
$\mathcal{F}^{\prime},$ provided that the foliation $\mathcal{F}_{|B}^{\prime
}$ is non-trivial and we will say that $\mathcal{F}^{\prime}$ is constructed
from $\mathcal{F}$ by adding a bubble to the leaf $L_{t_{0}}.$

We will define now the notion of bubble of $\mathcal{F}$ in a more general way.

\begin{definition}
Let $S$ be a subset of a singular leaf $L_{0}$ of $\mathcal{F}$ such that $S$
is homeomorphic to $S^{2}$ and $S$ contains a conic singularity $s.$ Let $W$
be the closure of the component of $S^{3}-S$ which does not contain $L_{0}-S$
and we assume that $\mathcal{F}_{|W}$ is non-trivial. The 3-ball $W$ will be
called a \emph{bubble} of $\mathcal{F}$ and the interior $Int(W)$ of $W$ will
be called \textit{the interior of the bubble. The 2-sphere }$S$ will be called
the boundary of $W$ and it will be said that $S$ defines the bubble $W.$
\end{definition}

\begin{figure}[ptb]
\begin{center}
\includegraphics[scale=0.71]
{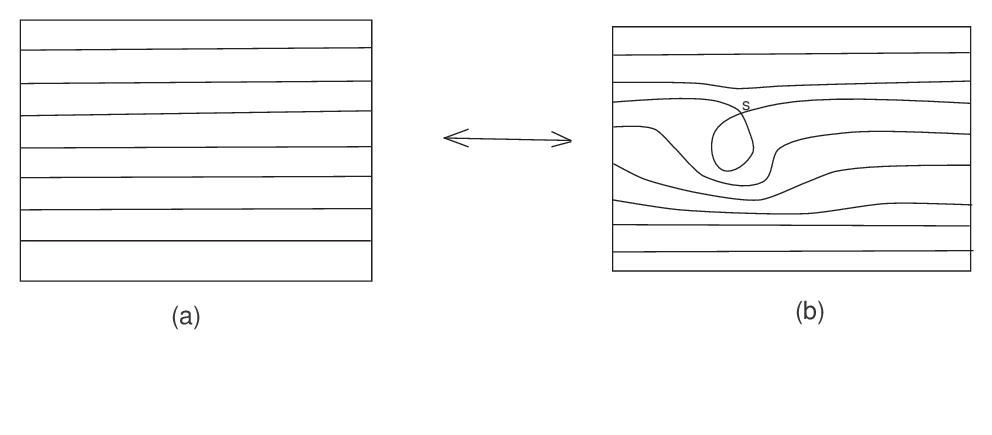}
\end{center}
\caption{Construction of bubbles in a Morse foliation.}%
\label{sxima5}%
\end{figure}

Now $W$ can be eliminated in two steps. First we foliate $Int(W)$ by
concentric 2-spheres by adding a center $c.$ We take in this way a foliation
$\mathcal{F}^{\prime}$ such that the pair of singularities $(s,c)$ forms a
trivial pair. Hence, in a second step, we can eliminate $(s,c)$ as it is
described above. If $\mathcal{F}^{\prime \prime}$ is the resulting foliation,
then we will say that $\mathcal{F}^{\prime \prime}$ results from $\mathcal{F}$
by eliminating (or removing) the bubble $W.$ Below by the term bubble we mean
always a non-trivial one.

It is worthy to notice that the procedure of eliminating the bubbles of
$\mathcal{F}$ creates a new Morse foliation $\mathcal{F}_{0}$, which is not
well defined in the sense that $\mathcal{F}_{0}$ depends on the order in which
the bubbles are eliminated. For instance, let $W_{1},$ $W_{2}$ be two bubbles
such that $S_{2}=\partial W_{2}$ is contained in the interior of $W_{1}$ and
$S_{1}=\partial W_{1}$ is contained in the interior of $W_{2}.$ If we
eliminate $W_{1}$ (resp. $W_{2})$ obtaining a foliation $\mathcal{F}_{1}$
(resp. $\mathcal{F}_{2})$ then the bubble $W_{2}$ (resp. $W_{1})$ disappears
also. Obviously $\mathcal{F}_{1}\neq$ $\mathcal{F}_{2}.$ However, we will show
that the order in which the bubbles are eliminated does not affect our main
results as they are stated in Theorems \ref{main theorem} and \ref{stability}.

The following proposition can be easily proven.

\begin{proposition}
\label{existence of Morse components}Let $\mathcal{F}$ be a Morse foliation
containing and assume that $\mathcal{F}$ does not have bubbles and admits at
least one non-spherical leaf. If $\mathcal{F}$ has at least one spherical leaf
then $\mathcal{F}$ contains a pseudo-toral leaf provided that each leaf of
$\mathcal{F}$ contains at most one conic singularity.
\end{proposition}

\begin{proof}
By Reeb stability theorem three cases can occur. Either $\mathcal{F}$ has two
centers and all its leaves are homeomorphic to $S^{2}$ or $\mathcal{F}$ has a
bubble or $\mathcal{F}$ has a pseudo-toral leaf $T.$ By hypothesis the first
and second cases are excluded. Therefore our result follows immediately.
\end{proof}

\medskip

\begin{figure}[ptb]
\begin{center}
\includegraphics[scale=0.86]
{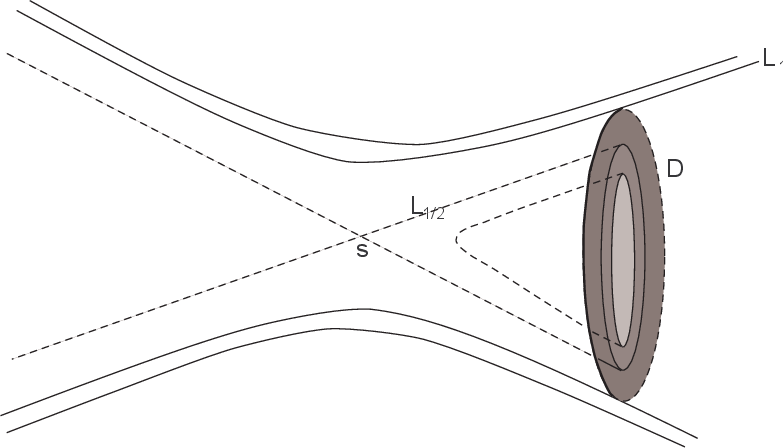}
\end{center}
\caption{A perfect disc associated to $s.$}%
\label{sxima6}%
\end{figure}

\subsection{Truncated bubbles\smallskip \medskip}

In the following we will introduce the notion of \textit{truncated bubble}.
The elimination of such bubbles is also necessary in order to simplify further
and understand better a Morse foliation.

We start with the following definitions.

\begin{definition}
\label{perfect position}Let $s$ be a conic singularity of $\mathcal{F}$ and
let $D$ be a 2-disc such that (see Figure 6):

(1) $\mathcal{F}_{|D}$ is a trivial foliation. The leaves of $\mathcal{F}%
_{|D},$ which are simple closed curves (s.c.c.), are denoted by $c_{t},$
$t\in(0,1].$

(2) If we denote by $L_{t}$ the leaf of $\mathcal{F}$ containing $c_{t},$
$t\in(0,1],$ then $c_{t}$ is the boundary of a 2-disc $D_{t}\subset L_{t}$
with $D_{t}\cap D=c_{t}$ for each $t\in(0,1/2].$ Furthermore, $s\in
Int(D_{1/2})$ while $D_{t}$ does not contain any conic singularity for
$t\in(0,1/2).$

(3) The curves $c_{t}$ are not homotopic to a constant in $L_{t}$ for each
$t\in(1/2,1].$

(4) The curve $c_{1}$ has trivial holonomy in its both sides.

If $D$ satisfies the above conditions (1)-(4) then $D$ will be called a
\textit{\  \emph{perfect disc associated to }}$s$ or simply a \emph{perfect
disc. }Also, $s$ will be called the\emph{\ singularity corresponding to the
perfect disc} $D.$

Hence, to each perfect disc corresponds a conic singularity and vice-versa.
\end{definition}

Now we are able to give the following definition.

\begin{definition}
\label{truncated bubble}Let $W$ be a 3-ball in $S^{3}$ and let $D_{i},$
$i=1,...,n$ be disjoint 2-discs in $\partial W$ such that:

(1) For each $i,$ $D_{i}$ is a perfect disc with corresponding conic
singularity $s_{i},$ such that $s_{i}\neq s_{j}$\textbf{\ }for\textbf{\ }%
$i\neq j.$

(2) If $F=\partial W-(\cup_{i}D_{i})$ then $F\mathcal{\ }$does not contain any
conic singularity.

(3) Assuming that all $s_{i}$ are contained in $W,$ there does not exist a
bubble $S$ in $W.$

Then $W$ will be called a truncated bubble and the set $\partial_{t}W=\partial
W-(\cup_{i}D_{i})$ will be called the tangent boundary of $W.$
\end{definition}

\begin{figure}[ptb]
\begin{center}
\includegraphics[scale=0.70]
{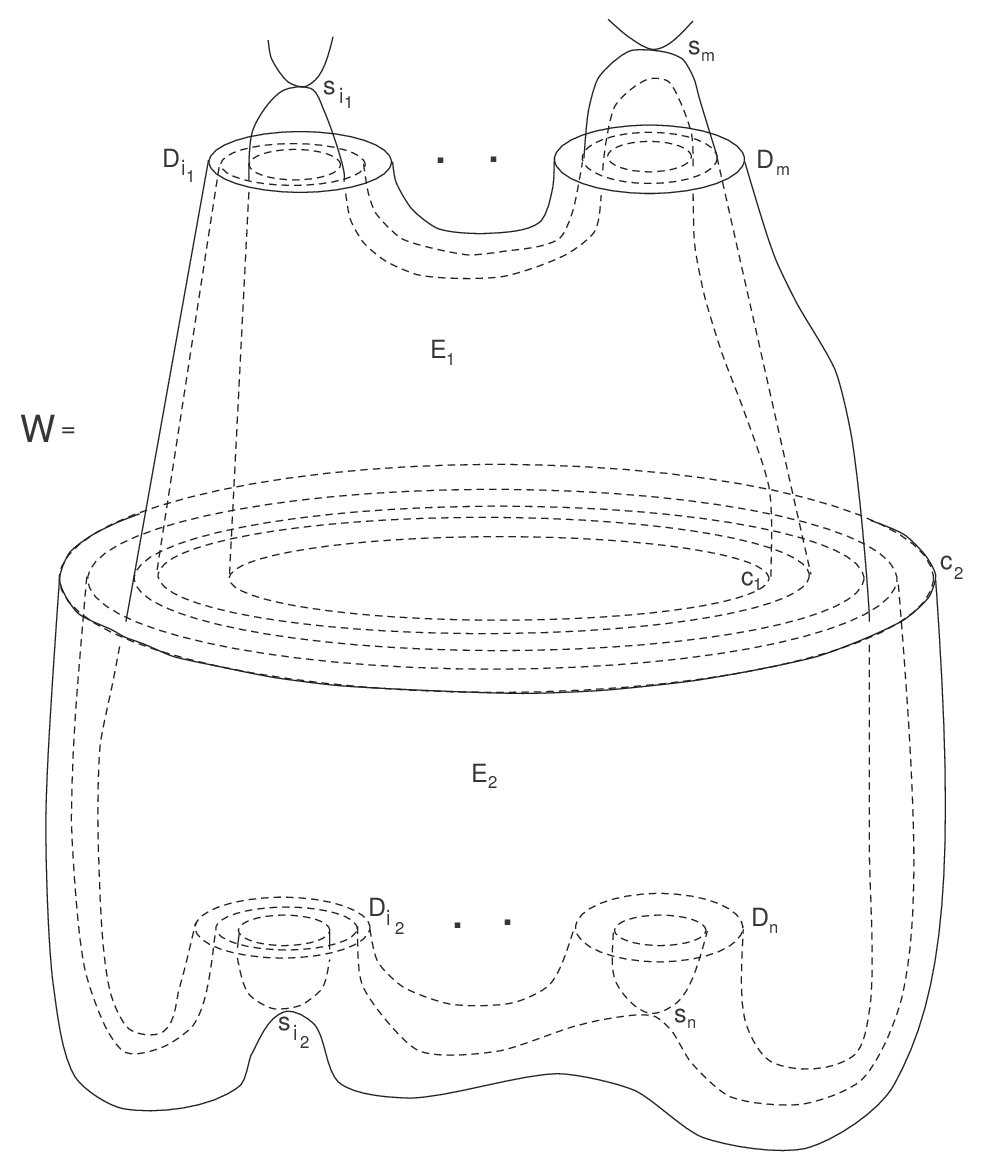}
\end{center}
\caption{A truncated bubble.}%
\label{sxima7}%
\end{figure}

Let $W$ be a truncated bubble and let $\overrightarrow{n}$ be the unit normal
vector to $\mathcal{F}$. The condition (3) in the Definition above ensures
that there exists at least one perfect disc $D_{k}\subset \partial W$ and one
perfect disc $D_{l}\subset \partial W$ such that $\overrightarrow{n}$ is
directed inward $D_{k}$ and outward $D_{l}.$ Indeed, if our statement was not
true, then we may deduce from Reeb's stability theorem \cite{[Reeb]} that a
bubble would exist in $W.$ Therefore the discs $D_{i}$ are separated in the
discs $D_{1},...,D_{m}$ and the discs $D_{m+1},...,D_{n},$ such that
$\overrightarrow{n}$ is directed inward on $D_{i}$ for $i=1,..,m$ and outward
on $D_{j}$ for $j=m+1,...,D_{n},$ see Figure 7.

Apparently, we may isotope each $D_{i}$ to a perfect disc, keeping $\partial
D_{i}$ on the same leaf, so that all $s_{i}$ are outside $W,$ see Figure 7.
Thus, in Definition \ref{truncated bubble}, we have assumed that all $s_{i}$
are contained in $W$ so that the hypothesis (3) of the definition makes sense.

In order to eliminate the truncated bubble $W$\ we may perform appropriate
Morse modifications around each conic singularity $s_{i}$ so that the
truncated bubble $W$\ disappears and in its place two bubbles $S_{1}$ and
$S_{2}$ are created. Notice that $S_{1}$ should be contained in the interior
of $S_{2}$ and $S_{2}$ should be contained in the interior of $S_{1}.$ Now, we
may eliminate the created bubbles as previously. That is, eliminating the
bubble $S_{1}$ (resp. $S_{2})$ the bubble $S_{2}$ (resp. $S_{1})$ is also
eliminated since $S_{2}$ (resp. $S_{1})$ is contained in the interior of
$S_{1}$ (resp. $S_{2}).$

In a next paragraph we will need the following property of a truncated bubble.
There is one disc $E_{1}$ which contains a conic singularity $s_{i_{1}},$
$1\leq i_{1}\leq m$ and one disc $E_{2}$ which contains a conic singularity
$s_{i_{2}},$ $m+1\leq i_{2}\leq n$ such that, if $\partial E_{1}=c_{1}$ and
$\partial E_{2}=c_{2}$ then the curves $c_{1}$ and $c_{2}$ bound an annulus
$C$ contained in a neighborhood of $\partial_{t}W$ and $\mathcal{F}_{|C}$ is a
family of parallel closed leaves, see Figure 7.

\begin{remark}
\label{elimination of all}From the discussion above, if $\mathcal{F}$ is a
Morse foliation then we may construct a new Morse foliation $\mathcal{F}%
^{\prime}$ which does not contain neither trivial pair of singularities nor
bubbles nor truncated bubbles. In fact, first we remove all trivial pair of
singularities, second we eliminate all bubbles and third we eliminate all
truncated bubbles. In the second and third step of our procedure we remark
that new bubbles/truncated bubbles are not created. So $\mathcal{F}^{\prime}$
results after finitely many steps.\medskip
\end{remark}

\subsection{Haefliger structures\medskip}

Finally, in the next few lines of this paragraph we will relate Haefliger
structures with Morse foliations in order to indicate that the latter are
naturally considered.

A $C^{r}$-Haefliger structure $\mathcal{H}$ on a manifold $M$ can be
considered as a generalized foliation on $M$ with a certain type of
singularities, see \cite{[Haefliger]}, \cite{[Lawson]} for more details.
Suppose that $\mathcal{H}$ is a codimension $q\ $ $C^{r}$-Haefliger structure.
Then, we may associate to $\mathcal{H}$, a $q$-dimensional $C^{r}$-vector
bundle $E$ over $M,$ a section $i:M\rightarrow E$ and a $C^{r}$-foliation
$\mathcal{G}$ defined in a neighborhood of $i(M)$ and transverse to the
fibers. The triple $(E,M,i)$ is called the graph of $\mathcal{H}$ and
determines it.

In particular, if we consider a Haefliger structure $\mathcal{H}$ of
codimension one on a closed 3-manifold $M$, then generically, the section
$i:M\rightarrow E$ is in general position with respect to $\mathcal{G}$, which
implies that the contact points are finitely many singularities of a Morse
function. In this way, a codimension 1, Morse foliation \ $\mathcal{F}$ on $M$
is associated to $\mathcal{H}$. Furthermore, it follows that in this
particular case, Morse foliations defined above are generic in the space of
Haefliger structures. This fact that makes the class of Morse foliations quite
interesting. \begin{figure}[ptb]
\begin{center}
\includegraphics[scale=0.66]
{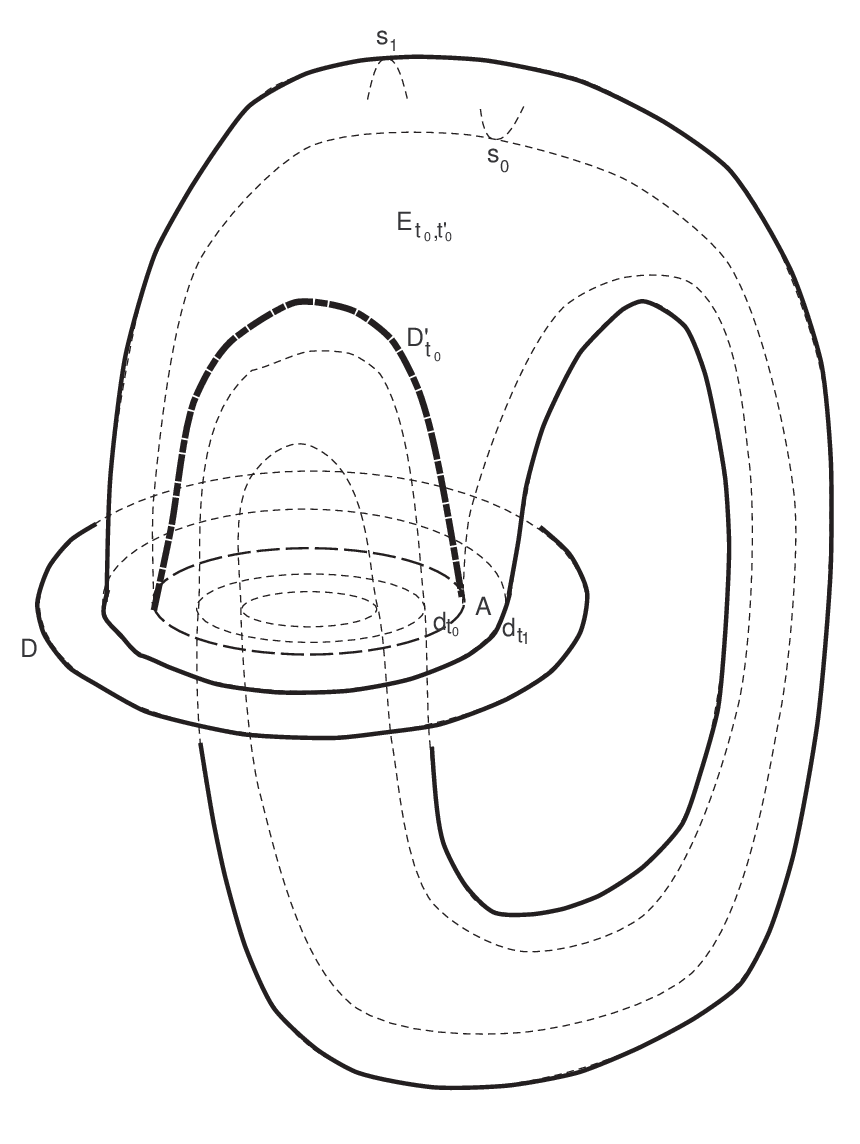}
\end{center}
\caption{Appearence of a truncated bubble.}%
\label{sxima8}%
\end{figure}

\section{Vanishing and anti-vanishing cycle}

In this section we recall the definition of vanishing cycle and we will
generalize Novikov's theorem for Morse foliations. In contrast to the
vanishing cycle the notion of \textit{anti-vanishing cycle} is introduced. We
always assume that each leaf contains at most one conic singularity and that
$\mathcal{F}$ does not contain trivial pair of singularities.

The following definition is borrowed from \cite{[Ros-Rou]} and is less general
than the initial one in \cite{[Novikov]} because we assume that the curves
$a_{t}$ below are simple.

\begin{definition}
\label{vanishing cycle}Let $f:S^{1}\times \lbrack0,1]\rightarrow S^{3}$ be a
smooth map such that:

(1) $f_{t}(S^{1})=a_{t}$ is a simple curve contained in a leaf $L_{t}$ for
each $t\in \lbrack0,1],$ where $f_{t}(x)=f(x,t);$

(2) for each $x\in S^{1},$ the arc $f_{t}(x)$ is transverse to $\mathcal{F};$

(3) the curve $a_{0}$ is not homotopic to a constant in $L_{0};$

(4) $a_{t}$ is nullhomotopic in $L_{t},$ $0<t\leq1.$

The curve $f_{0}(S^{1})=$ $a_{0}$ will be called a \emph{vanishing cycle} on
the leaf $L_{0}.$
\end{definition}

In order to state the next theorem we need the following definition.

\begin{definition}
\label{singular}A generalized Reeb component is a Morse foliation obtained by
adding finitely many bubbles to the leaves of a Reeb component.
\end{definition}

\begin{definition}
Let $\mathcal{R}$ be a Reeb component. If we add finitely many bubbles to the
leaves of $\mathcal{R}$ then we take a Morse foliation which will be called a
generalized Reeb component.
\end{definition}

Novikov in \cite{[Novikov]} proved that if $(M,\mathcal{F})$ is a closed,
orientable 3-manifold foliated by a regular, orientable foliation
$\mathcal{F}$ of codimension one and if some leaf of $\mathcal{F}$ contains a
vanishing cycle then $\mathcal{F}$ contains a Reeb component. Below we will
show that a similar result is also valid for Morse foliations of $S^{3}.$
Thus, we have:

\begin{theorem}
\label{Reeb component} Let $f:S^{1}\times \lbrack0,1]\rightarrow S^{3}$ be a
smooth map defining a vanishing cycle for the Morse foliation $\mathcal{F}$ of
$S^{3}.$ Then the foliation $\mathcal{F}$ contains a generalized Reeb component.
\end{theorem}

\begin{proof}
A first complete proof of this theorem for regular foliations is given in
\cite{[Ros-Rou]}. We follow it to prove Novikov's theorem for Morse
foliations. Below, we will sketch the steps of the proof.

Let $a_{0}=f(S^{1}\times \{0\})$ be the vanishing cycle defined by $f$ and let
$a_{t}$ be contained in a leaf $L_{t}$ for each $t.$ First of all, if $U_{i}$
is any bubble of $\mathcal{F}$ which does not contain $a_{0}$ in its interior
then we eliminate $U_{i}.$ In this way, the map $f$ of the theorem which
defines $a_{0}$ is not affected and furthermore, the lack of bubbles, which do
not contain $a_{0}$ in their interiors, allow us to apply the steps of proof
of Lemma 1 of \cite{[Ros-Rou]}. Thus, as in \cite{[Ros-Rou]}, the following
claim can be proved (see Lemma 1 in \cite{[Ros-Rou]}).

\textit{Claim 1}. There is an immersion $F:D^{2}\times(0,1]\rightarrow M$ such that,

(1) for each $t\in(0,1],$ $F_{t}(D^{2})$ is contained in a leaf $L_{t}$ and
$F_{t}$ is an embedding,

(2) $F_{t|\partial D_{t}}=f_{t},$

(3) for each $x\in D^{2}$, the curve $t\rightarrow F_{t}(x)$ is normal to the foliation.

As in \cite{[Ros-Rou]}, if this claim was not valid then a spherical leaf
would exist. Indeed, the set of $t\in(0,1]$ for which $F_{t}$ can be defined
to satisfy conditions (1)-(3) is open by Reeb's stability theorem. So, it
suffices to show that this set is also closed. Suppose $t_{0}\in(0,1)$ such
that $F_{t}$ is defined for $t_{0}<t\leq1$. The map $f_{t_{0}}:S^{1}%
\rightarrow L_{t_{0}}$ is an embedding and is nullhomotopic hence $f_{t_{0}}$
extends to an embedding $F_{t_{0}}:D^{2}\rightarrow L_{t_{0}}.$ Now displace
$F_{t_{0}}(D^{2})$ by the orthogonal trajectory field to the leaves $L_{t}$
for some $t>t_{0}.$ Let denote by $D_{t}$ this displaced disc in $L_{t}.$ Then
there are two possibilities, either $F_{t}(D^{2})=D_{t}$ or $F_{t}(D^{2})\cap
D_{t}=f_{t}(S^{1}).$ In the first case, $F$ is extended continuously on
$[t_{t_{0}},1].$ In the second case, by Proposition
\ref{existence of Morse components}, $L_{t}$ should be either a spherical leaf
or a pseudo-torus. But we have eliminated all bubbles which do not contain
$a_{0}$ in their interior. Hence $L_{t}$ cannot be a spherical leaf. On the
other hand, if $L_{t}$ was a pseudo-torus then for $\varepsilon>0$ arbitrarily
small the curve $f_{t+\varepsilon}(S^{1})$ should not be null homotopic in
$L_{t+\varepsilon}.$ This gives a contradiction which proves Claim 1.

The other steps of the proof, as they are expressed in Lemmata 2-6 of
\cite{[Ros-Rou]}, can be repeated verbatim and so Novikov's Theorem follows
for Morse foliations. Obviously, if we add again all eliminated bubbles we
obtain a generalized Reeb component.
\end{proof}

\begin{remark}
In the proof above, we have assumed that each $f_{t}$ is an embedding. For
regular foliations of $S^{3}$ the theorem is valid without this
extra-hypothesis \cite{[Novikov]}, \cite{[CaCo]}. Novikov's theorem could also
be proven in its full generality, but for the purpose of the present paper we
don't need the stronger version of the theorem.
\end{remark}

\begin{figure}[ptb]
\begin{center}
\includegraphics[scale=0.76]
{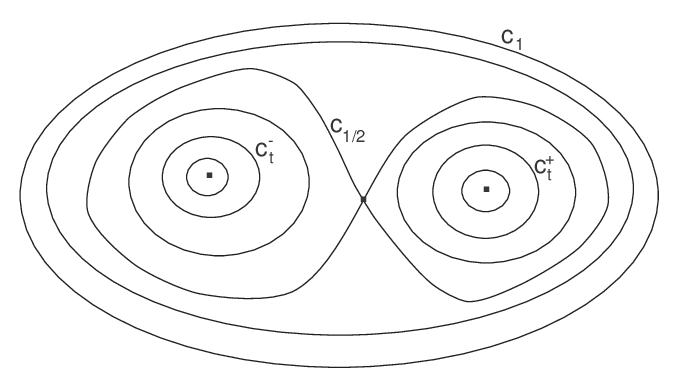}
\end{center}
\caption{The foliation induced on the disc $D.$}%
\label{sxima9}%
\end{figure}

\begin{definition}
\label{anti-vanishing}Let $f:S^{1}\times \lbrack0,1]\rightarrow S^{3}$ be an
embedding such that:

(1) $f_{t}(S^{1})=a_{t}$ is a simple curve contained in a leaf $L_{t},$ $0\leq
t\leq1,$ where $f_{t}(x)=f(x,t);$

(2) for each $x\in S^{1},$ the arc $f_{t}(x)$ is transverse to $\mathcal{F};$

(3) the curves $a_{t}$ are nullhomotopic in $L_{t},$ $0\leq t\leq1/2;$

(4) the curves $a_{t}$ are not homotopic to a constant in $L_{t},$
$1/2<t\leq1.$

The curve $f_{1/2}(S^{1})=$ $a_{1/2}$ will be called an $\emph{anti-vanishing}%
$ $\emph{cycle}$ on $L_{1/2}.$
\end{definition}

The following terminology will be used below.

Let $E$ be a 2-disc and $(p,q)$ be pair of points in the interior of $E.$ Let
$E^{\prime}$ be the space obtained by identifying $p$ with $q.$ Then
$E^{\prime}$ will be called a \textit{pseudo-disc} and the point $p\equiv q$
will be called a \textit{double point }of $E^{\prime}.$

We have the following lemma.

\begin{lemma}
\label{associated disc-gen Morse component} Let $f:S^{1}\times \lbrack
0,1]\rightarrow S^{3}$ be a smooth family of embedded circles $f_{t}(S^{1})$
which define an anti-vanishing cycle $f_{1/2}(S^{1})\subset$ $L_{1/2}.$ Then
either there exists a perfect disc $D$ associated to some conic singularity
$s$ such that all $f_{t}(S^{1}),$ $t\in \lbrack0,1]$ are leaves of
$\mathcal{F}_{|D}$ or, the appearance of $f_{1/2}(S^{1})$ implies the
existence of a truncated Morse component.
\end{lemma}

\begin{proof}
Let $f_{1/2}(S^{1})=a.$ Since $a$ is a simple closed curve, nullhomotopic in
$L_{1/2},$ it bounds a simply connected subset $\Delta \subset L_{1/2}$ which
is either a 2-disc or a pseudo-disc with a double point. Since $f_{t}(S^{1})$
is not nullhomotopic in $L_{t},$ $1/2<t\leq1$ we deduce the existence of a
conic singularity $s\in \Delta$ which in the latter case is the double point of
the pseudo-disc $\Delta.$ In the first case, and since we have assumed that
$L_{0}$ contains a single singularity (see Remark \ref{Morse modification}),
it follows that we may find a trivially foliated neighborhood $U$ of $s$ in
$S^{3}$ such that $\Delta \subset U.$ Therefore, there is in $U$ a perfect disc
$D$ associated to $s$ such that $f_{1/2}(S^{1})$ is a leaf of $\mathcal{F}%
_{|D}.$ In the second case, the existence of a pseudo-disc is equivalent with
the existence of truncated Morse/pseudo-Morse component.
\end{proof}

From the previous proof we have that each anti-vanishing cycle $a$ on a leaf
$L$ determines a conic singularity $s$ which will be referred to as the
\textit{singularity determined by the anti-vanishing cycle} $a.$ The
singularity $s$ is unique provided that each leaf contains at most one conic singularity.

Two anti-vanishing cycles $a,$ $a^{\prime}$ which determine the same conic
singularity will be referred to as \textit{equivalent}. Obviously $a$ and
$a^{\prime}$ must belong to the same leaf $L$ of $\mathcal{F}.$

Finally, the following terminology will be used in the next
paragraph:\smallskip

\noindent \textbf{Terminology}: Let $D$ be 2-disc embedded in $S^{3}$ which
does not contain any conic singularity of $\mathcal{F}.$ We assume that $D$ is
in general position with respect to $\mathcal{F}$, $\partial D$ belongs to a
leaf $L$ of $\mathcal{F}$ and $\partial D$ is a non-nullhomotopic curve in
$L.$ We assume that $a$ is an anti-vanishing cycle on a leaf $L_{a}$ of
$\mathcal{F}$ such that: $a$ is the boundary of a sub-disc $A\subset D$ and
all the leaves of $\mathcal{F}_{|A}$ are closed curves, nullhomotopic in the
leaves of $\mathcal{F}$ where they belong. In this case, we will say that the
anti-vanishing cycle $a,$ as well, any anti-vanishing cycle $a^{\prime}$
equivalent to $a,$ \textit{appears in }$D$\textit{ or that }$a,$ equivalently
$a^{\prime},$ \textit{ is contained in }$D.$ Obviously, there exists at least
one center $c$ of $\mathcal{F}_{|D}$ in the interior of $A$ surrounded by $a.$

\begin{figure}[ptb]
\begin{center}
\includegraphics[scale=0.66]
{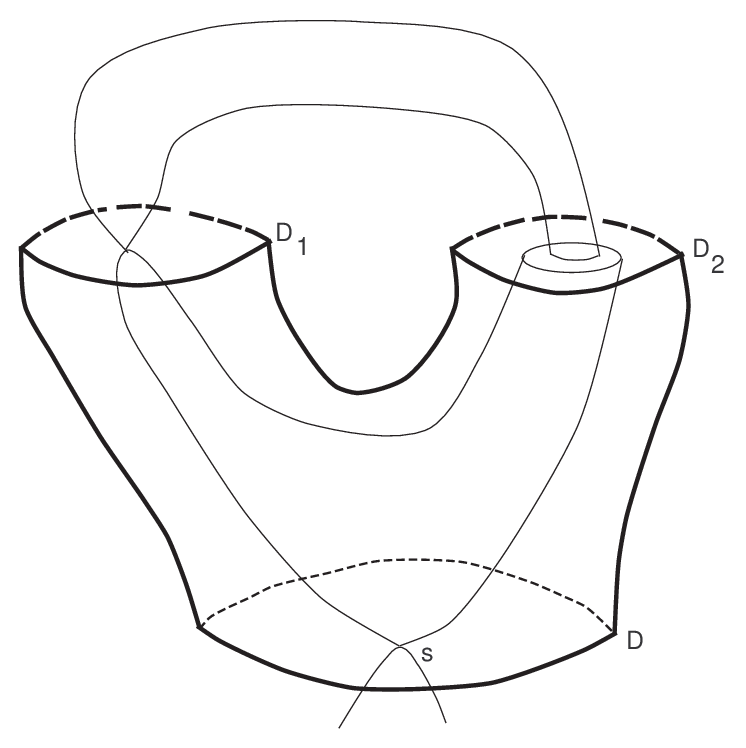}
\end{center}
\caption{A solid pair of pants with $s\neq s_{1}.$}%
\label{sxima10}%
\end{figure}

\section{\bigskip Existence of truncated Reeb components}

In this paragraph we make the following assumptions:

(1) $\mathcal{F}$ does not have all its leaves compact and simply connected;

(2) $\mathcal{F}$ does not have trivial pairs of singularities;

(3) $\mathcal{F}$ does not have bubbles or truncated bubbles;

(4) $\mathcal{F}$ does not have a vanishing cycle;

(5) $\mathcal{F}$ does not have a Morse component or a truncated Morse component.

From Remark \ref{elimination of all} such a foliation $\mathcal{F}$ can be
obtained after eliminating successively all trivial pair of singularities, all
bubbles and all truncated bubbles. Under these assumptions we will prove that
$\mathcal{F}$ has a truncated Reeb component. Notice that this component
appears in the Example 2 of Section 2. The basic idea for this proof is to
capture the conic singularities inside trivially foliated 3-balls and study
their complement.

The lack of bubbles and truncated bubbles give us a better insight into the
study of $\mathcal{F}$. From Proposition \ref{Reeb component}, the lack of
vanishing cycles implies the lack of Reeb/singular Reeb components. Also,
Novikov's theorem \cite{[Novikov]} implies that if $\mathcal{F}$ does not have
a vanishing cycle then $\mathcal{F}$ must have central or conic singularities.
Finally, from Lemma \ref{associated disc-gen Morse component}, the lack of
Morse/truncated Morse components implies that there do not exist pseudo-discs
with double points contained in some singular leaves of $\mathcal{F}$. Thus,
if $a$ is an anti-vanishing cycle on a leaf $L$, it follows that there exists
a 2-disc $D\subset L$ containing a conic singularity in its interior with
$\partial D=a.$

We need the following.

\begin{lemma}
\label{because of the lack of truncated bubbles} Let $D$ be an embedded disc
such that:

(i) $\mathcal{F}_{|D}$ is a trivial foliation of concentric circles
$d_{t}\subset L_{t},$ $t\in \lbrack0,1],$ where $L_{t}$ is a leaf of
$\mathcal{F}$ and $d_{0}$ degenerates to a point of $D;$

(ii) $d_{1}\nsim0$ in $L_{1}$ i.e. $d_{1}=\partial D$ is a non-nullhomotopic
curve in $L_{1}$.

Then there exists $t_{0}\in(0,1)$ such that: $d_{t}\sim0$ in $L_{t}$ for each
$t\in(0,t_{0}]$ i.e. all the curves $d_{t}$ are nullhomotopic in $L_{t} $ for
each $t\in(0,t_{0}],$ while $d_{t}\nsim0$ in $L_{t}$ for each $t\in(t_{0},1].$
\end{lemma}

\begin{proof}
Let $t_{0}=\sup \{t\in(0,1]:$ each $d_{t}$ is nullhomotopic in $L_{t}$ for
$t\leq t_{0}\}.$ Obviously a such $t_{0}$ exists because $\mathcal{F}$ does
not have a vanishing cycle on any of its leaves. Therefore there exist discs,
say $E_{t},$ such that: $\partial E_{t}=d_{t}$ and $E_{t}\subset L_{t}$ for
each $t\in(0,t_{0}].$ Let $D_{t}$ be the subdisc of $D$ bounded by $d_{t}$ for
each $t\in(0,1).$ Obviously $E_{t_{0}}\cap D$ consists of finitely many leaves
of $\mathcal{F}_{|D_{t_{0}}}.$ If $E_{t_{0}}\cap D\neq d_{t_{0}},$ we denote
by $d_{t_{0}^{\prime}}$ the leaf of $\mathcal{F}_{|Int(D_{t_{0}})}$ which is
closest to $d_{t_{0}}$ in $D_{t_{0}}.$ The disc $E_{t_{0}}$ must contain a
conical singularity, say $s_{0},$ see Figure 8. Notice that, if $E_{t_{0}}\cap
D\neq d_{t_{0}}$ and if $E_{t_{0},t_{0}^{\prime}}\subset E_{t_{0}}$ is the
annulus bounded by the curves $d_{t_{0}^{\prime}}$ and $d_{t_{0}},$ then
$s_{0}\in$ $E_{t_{0},t_{0}^{\prime}}$ otherwise $\mathcal{F}$ should have a bubble.

Now, let $t_{1}=\inf \{r\in(t_{0},1):d_{r}\sim0$ in $L_{r}$ and $d_{t}\nsim0$
in $L_{t}$ for each $t\in(t_{0},r)\}.$ Apparently, a such $t_{1}$ always
exists if the lemma is not valid. On the other hand, there exists a disc
$E_{t_{1}}\subset L_{t_{1}}$ with $\partial E_{t_{1}}=d_{t_{1}}$ and
$E_{t_{1}}$ must contain a conic singularity $s_{1}.$ We may see that
$s_{1}\neq s_{0}$ since $d_{t}\nsim0$ in $L_{t}$ for each $t\in(t_{0},t_{1}).$
Also, if we denote by $A=D_{t_{1}}-Int(D_{t_{0}})$ then the interior of
$E_{t_{1}}$ cannot intersect $A$ but it can intersect $Int(D_{t_{0}}),$ see
Figure 8.

\textit{Claim}: If there exists a $t_{1}\in(t_{0},1)$ as above, then
$\mathcal{F}$ has a truncated bubble.

\textit{Proof of Claim}. We may replace $D_{t_{0}}$ with a disc $D_{t_{0}%
}^{\prime}$ with $\partial D_{t_{0}}^{\prime}=d_{t_{0}}$ such that (see Figure 8):

(1) $\mathcal{F}_{|D_{t_{0}}^{\prime}}$ is trivial;

(2) $D_{t_{0}}\cap D_{t_{0}}^{\prime}=d_{t_{0}}$ and all the curves of
$\mathcal{F}_{|D_{t_{0}}^{\prime}}$ are contractible in the leaves of
$\mathcal{F}$ where they belong;

(3) $E_{t_{1}}\cap D_{t_{0}}^{\prime}=\varnothing$ and $E_{t_{0}}\cap
D_{t_{0}}^{\prime}=d_{t_{0}}.$

Let $S_{t_{1}}=E_{t_{1}}\cup A\cup D_{t_{0}}^{\prime}.$ Then $S_{t_{1}}$ is a
2-sphere and let $B_{t_{1}}$ be the 3-ball bounded by $S_{t_{1}}$ and
containing some neighborhood $V(s_{1})\subset L_{t_{1}}$ of $s_{1}$ which is
homeomorphic to a double cone.

In a trivial neighborhood $U$ of $s_{1}$ we may find first a simple closed
curve $e$ contained in some leaf $L$ of $\mathcal{F}_{|B_{t_{1}}}$and second,
an annulus $C\subset L\cap Int(B_{t_{1}})$ such that:

(i) $e$ is not nullhomotopic in $L;$

(ii) $\partial C=e\cup d_{t^{\prime}},$ where $t^{\prime}\in(t_{1}%
-\varepsilon,t_{1})$ for $\varepsilon$ sufficiently small and $d_{t^{\prime}}$
is a leaf of $\mathcal{F}_{|A}.$

Now we may consider a perfect disc $D_{e}\subset U$ with $\partial D_{e}=e$
and the perfect disc $D_{t^{\prime}}\subset D_{t_{1}}$ with $\partial
D_{t^{\prime}}=d_{t^{\prime}}.$ Obviously the sphere $D_{e}\cup C\cup
D_{t^{\prime}}$ is the boundary of a truncated bubble, say $W,$ contained in
$B_{t_{1}}.$ In this way we take a contradiction which proves our claim.

Finally, the claim above gives a contradiction to the assumption (3) and thus
our lemma is proven.
\end{proof}

We need the following definitions.

\begin{definition}
Let $L_{0},$ $L_{1}$ be two surfaces in $S^{3}$ such that:

1) $L_{0},$ $L_{1}$ are homeomorphic and each one is a subset of a leaf of
$\mathcal{F}$;

2) there exists a trivial foliated product $V=L\times \lbrack0,1]$ in $S^{3}$
such that, each $L\times \{t\}$ is a subset of a leaf of $\mathcal{F}$ with
$L\times \{0\}=L_{0},$ $L\times \{1\}=L_{1}.$

Then the surfaces $L_{0},$ $L_{1}$ will be called parallels in $\mathcal{F}$
and $\mathcal{F}_{|V}$ will be referred to as band of leaves of $\mathcal{F} $.
\end{definition}

\begin{definition}
Let $B$ be a 3-ball in $S^{3}$ and let $D_{i},$ $i=1,...,n$ be disjoint
2-discs in $\partial B$ such that:

\begin{itemize}
\item For each $i,$ $\mathcal{F}_{|D_{i}}$ is a trivial foliation.

\item For each $i,$ if $d_{i}=\partial D_{i},$ then $d_{i}$ is
non-contractible in the leaf of $\mathcal{F}$ where it belongs and it has
trivial holonomy.

\item If $S=\partial B-(\cup_{i}D_{i})$ then $S\mathcal{\ }$does not contain
any conic singularity.

\item After performing a finite number of Morse modifications, if necessary,
each leaf of $\mathcal{F}_{|B}$ is transformed to a leaf parallel to $S.$
\end{itemize}

The pair $(B,\mathcal{F}_{|B})$ will be referred to as a \emph{trivially
foliated ball with spots} $D_{i}.$ If $B$ has three spots then $B$ will be
referred to as a \emph{solid pair of pants}.

The set $\partial_{t}B=\partial W-(\cup_{i}D_{i})$ will be called the tangent
boundary of $B.$
\end{definition}

\begin{figure}[ptb]
\begin{center}
\includegraphics[scale=0.86]
{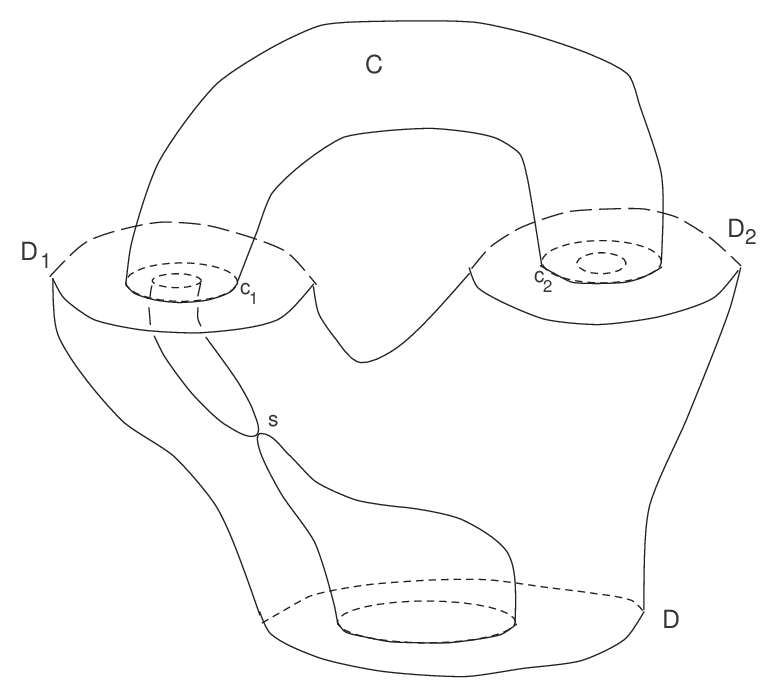}
\end{center}
\caption{A solid pair of pants with $s=s_{1}=s_{2}.$}%
\label{sxima11}%
\end{figure}

From Lemma \ref{because of the lack of truncated bubbles}, each spot of a
trivially foliated ball is a perfect disc. Therefore in the definition above
the three first conditions are the same with the conditions (1)-(3) of
Definition \ref{truncated bubble}. However, a trivially foliated ball with
spots is not a special case of a truncated bubble. The reason is that, if
$\overrightarrow{n}$ denotes the unit normal vector to $\mathcal{F}$, then
$\overrightarrow{n}$ is directed either inward or outward $D_{i}$ for each
$i.$ Also, the conic singularities which are associated to each spot $D_{i}$
are not necessarily distinct.

\begin{definition}
Let $B,$ $B^{\prime}$ be trivially foliated balls with spots. We will say that
$B$ and $B^{\prime}$ are equivalent if,

$(a)$ $B^{\prime}\subset B;$

$(b)$ $\partial_{t}B$ is parallel to $\partial_{t}B^{\prime}$ in $\mathcal{F}
$.

From now on we will consider classes $[B]$ of trivially foliated balls $B$ and
we will not distinguish two elements which belong in the same class. Thus, we
denote by $B$ the whole class $[B]$ if no misunderstandings are caused.
\end{definition}

\begin{definition}
Let $B$ be a 3-ball in $(S^{3},\mathcal{F})$ such that $(B,\mathcal{F}_{|B})$
is a trivially foliated ball with $n$ spots $D_{i}.$ Let $D_{1},$ $D_{2}$ be
two such spots in $\partial B$ and let $c_{1},$ $c_{2}$ be two closed leaves
of $\mathcal{F}_{|D_{1}},$ $\mathcal{F}_{|D_{2}}$ respectively. If there is an
annulus $C$ with $\partial C=c_{1}\cup c_{2}$ and if $C$ is contained in leaf
of $\mathcal{F}$, then we will say that the spot $D_{1}$ (resp. $D_{2})$ is
\emph{captured} or that the pair of spots $D_{1},$ $D_{2}$ are \emph{captured
between them}.

If a spot $D_{i}$ is not captured by another spot of $B$ then $D_{i}$ will be
called \emph{fre}$\emph{e}.$
\end{definition}

\begin{figure}[ptb]
\begin{center}
\includegraphics[scale=0.86]
{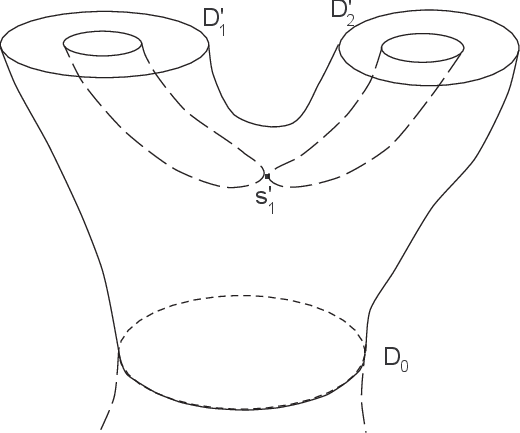}
\end{center}
\caption{In a solid pair of pants the hypothesis $s_{1}^{\prime}=s_{2}%
^{\prime}$ implies $s_{1}^{\prime}\neq s.$}%
\label{sxima12}%
\end{figure}Now we are able to prove the following proposition.

\begin{proposition}
\label{captured spot}Let $B$ be a trivially foliated ball with spots $D_{i},$
$i=1,..,k.$ If two of the spots of $B$ are captured between them then
$\mathcal{F}$ has a truncated Reeb component.
\end{proposition}

\begin{proof}
We assume that the spots $D_{1},$ $D_{2}$ are captured between them. Then we
may find a disc $D$ embedded in $B$, in general position with respect to
$\mathcal{F}$ such that:

1) $\mathcal{F}_{|D}$ has the form of Figure 9;

2) $\partial D,$ $\partial D_{1},$ $\partial D_{2}$ are the boundaries of a
pair of pants $P\subset \partial_{t}B\subset L_{0},$ where $L_{0}$ is a leaf of
$\mathcal{F}$.

Thus, in a standard way, we may isotope $D$ in $B$ such that $\mathcal{F}%
_{|D}$ is trivial.

Let denote by $B_{0}\subset B$ the 3-ball with $\partial B_{0}=P\cup D_{1}\cup
D_{2}\cup D.$ On each disc $D$, $D_{1},$ $D_{2}$ there is respectively an
anti-vanishing cycle, say $a,$ $a_{1},$ $a_{2},$ because by hypothesis, there
do not exist vanishing cycles. From the proof of Lemma
\ref{associated disc-gen Morse component} the anti-vanishing cycle $a$ (resp.
$a_{1},$ $a_{2})$ determines a conic singularity $s$ (resp. $s_{1},$ $s_{2}).$

Now we claim that $s=s_{1}=s_{2}.$ In fact, assuming that $s\neq s_{1}$ we
will get a contradiction. To prove it, we may assume that $s\in D,$ $s_{1}\in
D_{1}$ and thus we have the configuration of Figure 10. But then $\mathcal{F}$
must have a bubble which is impossible from the assumption (3) at the
beginning of the section.

Thus, we have that $s=s_{1}=s_{2}$ and without loss of generality, we may
assume that we have the configuration of Figure 11. Since the spots $D_{1},$
$D_{2}$ are captured between them, there exists an annulus $C$ with $\partial
C=c_{1}\cup c_{2}$ such that: $C\cap B_{0}=c_{1}\cup c_{2}$ where $c_{i}$ is a
leaf of $\mathcal{F}_{|D_{i}},$ $i=1,2$ and each $c_{i}$ is a non-contractible
curve in the leaf $L$ of $\mathcal{F}$, where it belongs. We denote by
$D_{i}^{\prime}$ the subdisc of $D_{i}$ bounded by $c_{i},$ $i=1,2.$ Then
$S=C\cup D_{1}^{\prime}\cup D_{2}^{\prime}$ is a 2-sphere and we denote by $A$
the 3-ball bounded by $S$ such that $A\cap B_{0}=D_{1}^{\prime}\cup
D_{2}^{\prime}.$ Since $\mathcal{F}_{|D_{i}^{\prime}}$ is trivial we have that
all the leaves in the interior of $A$ are either annuli or discs and no leaf
of $\mathcal{F}_{|A}$ contains a conic singularity.

On the other hand, by applying a Morse modification around $s$ the pair of
pants $P$ can be transformed to a cylinder $P_{1}$ such that:

In $P_{1}$ the curve $\partial D$ is contractible in the leaf of $\mathcal{F}$
where it belongs while the curves $\partial D_{1},$ $\partial D_{2},$ as well
as, the curves $c_{1}$ and $c_{2}$ are not nullhomotopic in the leaves of
$\mathcal{F}$ where they belong respectively.

Let $B_{1}$ be the 3-ball with $\partial B_{1}=P_{1}\cup D_{1}\cup D_{2}$ and
$B_{1}\cap A=D_{1}^{\prime}\cup D_{2}^{\prime}.$ If we denote by
$\mathcal{F}_{1}$ the Morse foliation obtained after performing the previous
Morse modification, we deduce by our construction, that each leaf of
$\mathcal{F}_{1|B_{1}}$ is also non-singular.

Examining now the foliation $\mathcal{F}_{1|D_{1}^{\prime}}$ (or
$\mathcal{F}_{1|D_{2}^{\prime}})$ we derive the existence of a leaf $a$ of
$\mathcal{F}_{1|D_{1}^{\prime}}$ which defines a vanishing cycle on some leaf,
say $L_{0}$ of $\mathcal{F}.$ This follows from the following two facts: (1)
all the leaves of $\mathcal{F}_{1|A\cup B_{1}}$ are non-singular and (2) any
leaf $L^{\prime}$ of $\mathcal{F}_{|A}$ homeomorphic to a disc is a part of a
leaf $L^{\prime \prime}$ of $\mathcal{F}_{1}$ which stays always in $A\cup
B_{1}.$ The latter follows easily from Lemma
\ref{because of the lack of truncated bubbles}. Finally, from Proposition
\ref{Reeb component} the vanishing cycle $a$ implies the existence of a Reeb
component for $\mathcal{F}_{1}$ and thus $\mathcal{F}$ has a truncated Reeb component.
\end{proof}

We will describe now an algorithm which allow us to construct truncated Reeb components.

Let $D_{0}$ be a perfect disc and let $c_{0}=\partial D_{0}.$ We distinguish
the following two cases:

(I) There exists a solid pair of pants $B$ with three spots $D_{0},$ $D_{1}, $
$D_{2}.$

(II) There does not exist such a solid pair of pants $B.$

We proceed as follows:

In the case (II) we do nothing and our procedure stops.

In the case (I) two of the spots $D_{0},$ $D_{1},$ $D_{2}$ could be captured
between them but at least one will be free. Without loss of generality, we
assume that $D_{1}$ is a free spot and we repeat the some procedure with
$D_{1}$ in the place of $D_{0}.$ That is, for the free spot $D_{1}$ we examine
if there is a solid pair of pants $B^{\prime}$ with three spots $D_{1},$
$D_{1}^{\prime},$ $D_{2}^{\prime}$ such that: $B^{\prime}\cap B=D_{1}.$

If such a $B^{\prime}$ exists, then we set $B_{1}=B\cup B^{\prime}$ and thus a
trivially foliated 3-ball $B_{1}$ with 4 spots in $S^{3}$ is obtained. In this
case we will say that $B$ can be developed \textit{by adding a solid pair of
pants}. The spots of $B_{1}$ are the discs $D_{0},$ $D_{2},$ $D_{1}^{\prime},$
$D_{2}^{\prime}.$ This procedure will be referred to as the \textit{addition
of a trivially foliated solid pair of pants to }$B$\textit{\ along }$D_{1}$
and the free spot $D_{1}$ will be called an \textit{inessential free spot}.
Otherwise, $D_{1}$ will be called \textit{essential free spot}.

In the following, we distinguish two cases: either our procedure can be
continued, which means that we may add one more trivially foliated solid pair
of pants to $B_{1}$ along some essential free spot of $B_{1}$ or, $B_{1}$
cannot be developed further.

Let $B$ be a trivially foliated ball with spots. Then from Lemma
\ref{because of the lack of truncated bubbles} each spot of $B$ is a perfect
disc. Therefore, to each spot of $B$ corresponds a conic singularity. Let
$s_{i},$ $i=1,2,..,k$ be the conic singularities which correspond to the spots
of $B.$ Without loss of generality we may assume that each $s_{i}$ belongs in
the interior of $B.$ These singularities $s_{i}$ will be called conic
singularities of $B.$

We need the following lemma.

\begin{lemma}
\label{adding solid pair of pants} Let $B$ be a trivially foliated 3-ball in
$S^{3}$ with $n$ spots. We assume that $n_{1}$ of them are free and $n_{2}$
are captured. If we add to $B$ a trivially foliated solid pair of pants then
we take a 3-ball $B^{\prime}$ with, either $n_{1}-1$ free spots or with
$n_{1}-3$ free spots or $B^{\prime}$ has one more conic singularity with
respect to $B.$
\end{lemma}

\begin{proof}
By adding a trivially foliated solid pair of pants $B_{0}$ to $B$ along a free
spot of $B$, say $D_{0}$, we obtain a new 3-ball $B^{\prime}$ with two more
spots, say $D_{1}^{\prime},$ $D_{2}^{\prime}.$ If only one of $D_{1}^{\prime
},$ $D_{2}^{\prime}$ is captured then $B^{\prime}$ has necessarily $n_{1}-1$
free spots. Indeed, assuming for example that $D_{1}^{\prime}$ is a captured
spot and that $D_{2}^{\prime}$ is free, we deduce the existence of a free spot
of $B$, say $D,$ such that $D_{1}^{\prime}$ is captured by $D.$ Therefore, we
have destroyed two free spots of $B$: first the spot $D$ and second the spot
$D_{0}$ by adding $B_{0}.$ On the other hand, $B^{\prime}$ has a new free
spot, with respect to $B,$ the spot $D_{2}^{\prime}.$ Therefore $B^{\prime}$
has $n_{1}-1$ free spots. If both $D_{1}^{\prime},$ $D_{2}^{\prime}$ are
captured between them, then $B^{\prime}$ has again $n_{1}-1$ free spots
because the free spot $D_{0}$ of $B$ is destroyed. If both $D_{1}^{\prime},$
$D_{2}^{\prime}$ are captured but each one is captured with a spot of $B$ then
$B^{\prime}$ has $n_{1}-3$ captured spots, since we have destroyed three free
spots of $B.$

We assume now that both the spots $D_{1}^{\prime},$ $D_{2}^{\prime}$ are free.
Then on each $D_{i}^{\prime}$ there is a leaf $c_{i}$ of $\mathcal{F}%
_{|D_{i}^{\prime}}$ which defines an anti-vanishing cycle. Therefore, from
Lemma \ref{associated disc-gen Morse component}, there is a disc $\Delta_{i}$
such that: $\partial \Delta_{i}=c_{i},$ $\Delta_{i}$ is contained in a leaf
$L_{i}$ of $\mathcal{F}$ and $\Delta_{i}$ contains a conic singularity
$s_{i}^{\prime}.$ We have the followings:

(i) For each spot, say $D^{\prime},$ of $B^{\prime}$ it is $\Delta_{i}\cap
D^{\prime}=\varnothing.$

\noindent Indeed, if not, the spot $D_{i}^{\prime}$ should not be free.

(ii) If $s_{1}^{\prime}=s_{2}^{\prime}$ then $s_{1}^{\prime}\neq s$ for each
conic singularity $s$ of $B.$

\noindent Indeed, this case leads to the configuration of Figure 12.
Obviously, $s_{1}^{\prime}\neq s$ for each conic singularity $s$ of $B.$

(iii) If $s_{1}^{\prime}\neq s_{2}^{\prime}$ then some $s_{i}^{\prime}$ should
coincide with some conic singularity $s$ of $B.$

\noindent Indeed, if neither $s_{1}^{\prime}$ nor $s_{2}^{\prime}$ coincides
with a conic singularity $s$ of $B$ then $\mathcal{F}$ should contain a
spherical leaf. Now, from Proposition \ref{existence of Morse components},
since threre are no bubbles, the existence of a spherical leaf implies the
existence of a pseudo-toral leaf $T.$ Furthernore, since the 3-ball
$B^{\prime}$ is trivially foliated, the existence of $T$ implies the existence
of a Morse/truncated Morse component. But this is not allowed by our
hypothesis. Therefore our lemma is proven.
\end{proof}

A trivial foliated ball $B$ with spots will be called \textit{completely
expanded} if all the free spots of $B$ are essential.

Let now $B,$ $B^{\prime}$ be two disjoint trivially foliated balls with spots
and let $D,$ $D^{\prime}$ be two spots on $B,$ $B^{\prime}$ respectively. We
will say that $B,$ $B^{\prime}$ \textit{can be added through the spots }$D,$
$D^{\prime},$ if there are leaves $c,$ $c^{\prime}$ of $\mathcal{F}_{|D},$
$\mathcal{F}_{|D^{\prime}}$ respectively such that: $c,$ $c^{\prime}$ belong
on the same leaf $L$ of $\mathcal{F}$ and there exists an annulus $A\subset L$
with $\partial A=c\cup c^{\prime}.$ Otherwise, we will say that $B$ and
$B^{\prime}$ cannot be added.

In the case that $B,$ $B^{\prime}$ can be added through the spots $D,$
$D^{\prime}$ then, after performing a finite number of Morse modifications, a
new trivially foliated ball $B_{0}$ is created and we will say that $B_{0}%
$\textit{\ contains both }$B$\textit{\ and }$B^{\prime}.$ Therefore, modulo
Morse modifications, we deduce from Lemma \ref{adding solid pair of pants},
that there exists a maximum number of trivially foliated balls with spots, say
$B_{1},...,B_{k},$ which are disjoint among them, i.e. $B_{i}\cap
B_{j}=\varnothing$ for $i\neq j,$ and none of them can be added to another, in
other words, for each $i,$ all the free spots of $B_{i}$ are essential.

For each $i,$ we assume that $B_{i}$ has $n_{i}$ essential free spots.
Henceforth we will denote by $\mathcal{B}=\{B_{i}\}$ this family.

Now we are able to prove the following theorem.

\begin{theorem}
\label{existence Morse component}The Morse foliation $\mathcal{F}$ contains a
truncated Reeb component.
\end{theorem}

\begin{proof}
In order to fix the notation, if $N$ is a compact manifold the set
$C^{r}(N,S^{3})$ of $C^{r}$ maps from $N$ to $S^{3}$ for $r\geq2,$ will be
equipped with the strong topology, (see \cite{[Hirsch]}, Ch. 2). Roughly
speaking, a neighborhood of some $f\in C^{r}(N,S^{3})$ consists of all maps
$g\in C^{r}(N,S^{3})$ which are close to $f$ together with their derivatives
of order $k,$ $k=0,\cdots,r,$ at each point of $N.$ A neighborhood of $f$
consisting of all such $g$ will be denoted by $U_{\varepsilon}(f(N))$ or by
$U_{\varepsilon}(f).$ The sphere $S^{3}$ is also equipped with its standard metric.

By performing Morse modifications on $\mathcal{F}$ we obtain a foliation
$\mathcal{F}_{0}$ for which there exists the family $\mathcal{B}=\{B_{i}\}$
defined above, that is, each $B_{i}$ is a trivially foliated, completely
expanded ball with spots and $B_{i}\cap B_{j}=\varnothing$ for $i\neq j.$ We
may also assume that each conic singularity of $\mathcal{F}$ is contained in
the interior of some $B_{i}.$ Obviously each $B_{i}$ has at least two spots.
Our hypothesis implies that there do not exist disjoint elements $B_{i},$
$B_{j}\in$ $\mathcal{B}$ which can be added. We claim that the existence of
two captured spots, say $D,$ $D^{\prime},$ on the same $B_{i}$ implies the
existence of a truncated Reeb component for $\mathcal{F}$. To prove our claim,
we deduce from Proposition \ref{captured spot} the existence of a truncated
Reeb component, say $C_{0},$ for the foliation $\mathcal{F}_{0}$ and we will
show that $\mathcal{F}$ should also have a truncated Reeb component. Indeed,
in the proof of Proposition \ref{captured spot}, which ensures the existence
of $C_{0},$ a solid pair of pants $B$ is depicted in some $B_{i}$ with the
following features: \ $B$ has two spots captured between them and there exists
a common conic singularity $s$ associated to each spot of $B.$ Obviously, the
inverse Morse modifications, which lead $\mathcal{F}_{0}$ back to
$\mathcal{F}$ affect $C_{0},$ if and only if these Morse modifications are
applied around $s.$ But even in this case, it is clear that $C_{0}$ is
transformed to a truncated Reeb component $C$ of $\mathcal{F}$.

Therefore, in order to prove the theorem we will show that the existence of
the family $\mathcal{B}=\{B_{i}\}$ and the hypothesis that all the spots of
$B_{i}$ are essential free spots leads to a contradiction.

Let denote by $D_{1}^{i},...,D_{n_{i}}^{i}$ the spots of $B_{i},$
$i=1,2,..,k,$ and by $a_{j}^{i}$ the anti-vanishing cycle which appears in the
spot $D_{j}^{i}.$ We pick an arbitrary element of the family $\mathcal{B}$,
say $B_{1}.$ Let $W_{1}=S^{3}-Int(B_{1}),$ $S_{1}=\partial W_{1}=\partial
B_{1}$ and $c_{i}=\partial D_{i}^{1}.$

We consider a $C^{\infty}$-embedding $H:D^{2}\times \lbrack0,1]\rightarrow
W_{1}$ such that:

(1) $H(\partial D^{2}\times \lbrack0,1])\subset S_{1}$ and $H(Int(D^{2}%
)\times \lbrack0,1])\subset Int(W_{1}).$

(2) If $h_{t}=H(\cdot,t),$ the disc $D_{t}=h_{t}(D^{2})$ is transverse to
$S_{1}$ along $d_{t}=\partial D_{t}$ for each $t\in \lbrack0,1].$

Obviously, $D_{t}$ separates $W_{1}$ and the orientation of $[0,1]$ permit us
to talk about the left hand side and the right hand side of $D_{t}$ in
$W_{1}.$

Furthermore, we may assume that the embedding $H$ satisfies:

(3) Each $d_{t}$ contains on the left hand component of $S_{1}-d_{t}$ only the
curve $c_{1};$

(4) The disk $D_{0}$ is in general position with respect to $\mathcal{F}$ and
there is a unique anti-vanishing cycle $a_{0}\subset D_{0}$ i.e. $a_{0}$ is
contained in $D_{0}$ and it is equivalent to the anti-vanishing cycle
$a_{1}^{1}.$ Also, the disk $D_{1}$ is in general position with respect to
$\mathcal{F}$ and every anti-vanishing cycles $a_{j}$ contained in $D_{1}$ is
equivalent to the anti-vanishing cycles $a_{j}^{1},$ $j=2,..,n_{1}$ respectively.

The property (4) above can be obtained by choosing $D_{0}$ (resp. $D_{1})$ in
a sufficient thin tubular neighborhood $U$ of $\partial W_{1}$ in $W_{1}$ with
respect to the standard metric of $S^{3}.$

Obviously the following statement is valid. With respect to the strong
topology defined above, there is $\varepsilon_{0}>0$ such that: each disc
$E_{0}$ (resp. $E_{1}),$ which is $\varepsilon_{0}-$close to $D_{0}$ (resp. to
$D_{1})$ is in general position with respect to $\mathcal{F}$ and each
anti-vanishing cycle appeared in $E_{0}$ (resp. $E_{1})$ is equivalent to the
anti-vanishing cycle $a_{0}$ (resp. to some anti-vanishing cycle $a_{i}^{1}).$

Now, let

\begin{itemize}
\item $\mathcal{I}=\{t\in \lbrack0,1]$ for which there exists an $\varepsilon
_{t}>0$ such that: if $E_{t}$ is a 2-disc in general position with respect to
$\mathcal{F}$ and $\varepsilon_{t}-$close to $D_{t}$ with $\partial
E_{t}\subset S_{1},$ then in $E_{t}$ appears a single anti-vanishing cycle
which is equivalent to $a_{0}\}.$
\end{itemize}

Since we have assumed that all the spots of any $B_{i},$ and hence of $B_{1},$
are free and essential we deduce that $\mathcal{I}\neq \varnothing$ and if
$t_{0}=\max \mathcal{I}$ then $t_{0}<1.$

\textit{Claim 1}: $t_{0}\notin \mathcal{I}$.

\textit{Proof of Claim 1}. We assume that $t_{0}\in \mathcal{I}$. Then we
consider a $t^{\prime}$ on the right hand side of $t_{0}$ such that
$D_{t^{\prime}}$ is $\varepsilon_{t_{0}}/2$ close to $D_{t_{0}}.$ It follows
that, all the discs $E_{t}$ which are in general position with respect to
$\mathcal{F}$ and $\varepsilon_{t_{0}}/4-$close to $D_{t^{\prime}}$ have a
single anti-vanishing cycle equivalent to $a_{0}.$ This contradicts the
definition of $t_{0}$ as maximum of $\mathcal{I}$ and proves Claim
1.\smallskip

\textit{Claim 2}: There is a neighborhood $U_{\varepsilon}(h_{t_{0}})$ of
$h_{t_{0}}$ (denoted also by $U_{\varepsilon}(D_{t_{0}}))$ such that: if
$g:D^{2}\rightarrow S^{3}$ is an embedding in general position with respect to
$\mathcal{F}$ with $g(D^{2})\in U_{\varepsilon}(D_{t_{0}})$ and $g(\partial
D^{2})\subset \partial B_{1}$ then all the anti-vanishing cycles which appear
in $g(D^{2})$ are equivalent.

\textit{Proof of Claim 2}. In this proof we need the following terminology. If
$V$ is a trivial neighborhood of some conic singularity of $\mathcal{F}$ we
will denote by $\partial_{t}\overline{V}$ the two perfect discs in the
boundary $\partial \overline{V}$ of the closure $\overline{V}$ of $V.$ We
consider now a finite number of open subsets $U_{i}$ of $S^{3}$ which cover
$D_{t_{0}}$ such that: each $U_{i}$ is either a local chart of $\mathcal{F}$
or a trivial neighboord of some conic singularity of $\mathcal{F}$ and
$D_{t_{0}}$ intersects each $U_{i}.$ Furthermore, since the set of
singularities of $\mathcal{F}$ is finite we may assume that, if $U_{k}$ is any
trivial neighborhood of a conic singularity then $\partial_{t}\overline{U_{k}%
}\cap U_{j}=\emptyset$ for each $j\neq k.$ Let $U=\cup_{i}U_{i}.$ Obviously,
we may choose an $\varepsilon>0$ such that if $g,$ $g^{\prime}\in
U_{\varepsilon}(D_{t_{0}})$ then the discs $g(D^{2})$ and $g^{\prime}(D^{2})$
are contained in $U$ and both intersect each $U_{i}.$ From the properties of
the cover $\{U_{i}\}_{i}$ and in particular, since $\partial_{t}%
\overline{U_{k}}\cap U_{j}=\emptyset$ for each trivial neighborhood $U_{k}$
and each $U_{j}$ with $j\neq k,$ we deduce that the anti-vaniching cycles
which appear in $g(D^{2})$ and $g^{\prime}(D^{2})$ are equivalent. In fact,
for each trivial neighborhood $U_{k},$ the foliation $\mathcal{F}_{|U_{k}}$
induces the same anti-vanishing cycle on $g(D^{2})$ and on $g^{\prime}%
(D^{2}).$ This proves Claim 2.

Now from Claim 2, we may find positive numbers $\varepsilon_{t}$ and
$\varepsilon_{t^{\prime}}$ and discs $E_{t}\in U_{\varepsilon_{t}}(D_{t}),$
$E_{t^{\prime}}\in U_{\varepsilon_{t^{\prime}}}(D_{t^{\prime}})$ on the left
and the right side of $D_{t_{0}}$ respectively, which are in general position
with respect to $\mathcal{F}$ and which have equivalent anti-vanishing cycles.
Therefore, we get a contradiction to the definition of the set $\mathcal{I}.$
On the other hand, we may always assume that all the free spot of $B_{1}$ (as
well as, of any $B_{i})$ are essential. Therefore all the spots of $B_{1}$
cannot be free. In other words, there are spots of $B_{1}$ which are captured
between them. Therefore from Proposition \ref{captured spot} we get our theorem.
\end{proof}

\section{Proof of the main Theorem}

Let $\mathcal{F}$ be a Morse foliation on $S^{3}$ allowed to have bubbles,
truncated bubbles, as well as, vanishing cycles.

\begin{theorem}
\label{corollary2} Assuming that $\mathcal{F}$ does not have trivial pair of
singularities then $\mathcal{F}$ has either a Morse/truncated Morse component
or a Reeb/truncated Reeb component. Therefore, modulo a Morse modification,
$\mathcal{F}$ has either a Morse or a Reeb component.
\end{theorem}

\begin{proof}
From Theorem \ref{Reeb component}, we may assume that $\mathcal{F}$ does not
have Morse or truncated Morse components and vanishing cycles. If
$\mathcal{F}$ does not also have bubbles and truncated bubbles the proof
follows from Theorem \ref{existence Morse component}. If $\mathcal{F}$ has
bubbles but not truncated bubbles then we pick an innermost bubble $B;$ that
is, in the interior of $B$ there does not exist another bubble of
$\mathcal{F}$. Now, we may apply our method of work in $B$ and prove, as in
the previous paragraph, that in $B$ there is a Reeb/truncated Reeb component.

From now on we assume that $\mathcal{F}$ has also truncated bubbles. Thus, if
$W$ is a truncated bubble of $\mathcal{F}$ then the configuration of Figure 7
appears and the notation of paragraph of Section 2 will be used. Therefore,
for $W$ we may modify locally $\mathcal{F}$ in a neighborhood of annulus $C$
so that the disc $E_{1}$ fits with the disc $E_{2}$ so that $c_{1}=c_{2}.$
Applying for any truncated bubble $W$ a such modification on $\mathcal{F}$,
which will be referred to as a \textit{corrective movement}, we get a new
foliation, say $\mathcal{F}^{\prime}.$ It is obvious now that $\mathcal{F}%
^{\prime}$ does not have truncated bubbles. The union $E_{1}\cup E_{2}$ is
homeomorphic to $S^{2}$, it contains two conic singularities, say $s_{1},$
$s_{2},$ and will be referred to as a \textit{special bubble with two conic
singularities}. Therefore after applying finitely many, arbitrary small Morse
modifications around $s_{1}$ or $s_{2}$ we may assume that each leaf contains
at most one singularity. Let denote by $\mathcal{F}^{\prime \prime}$ the
obtained foliation after applying these last Morse modifications on
$\mathcal{F}^{\prime}.$

Now let $B_{1}$ be an innermost bubble of $\mathcal{F}^{\prime \prime}.$ In
$B_{1}$ we may prove that there exists Reeb/truncated Reeb component as
before. If $B_{1}$ existed as an innermost bubble of $\mathcal{F}$ then we
deduce that $\mathcal{F}$ has a Reeb/truncated Reeb component. Hence we assume
that $B_{1}$ is an innermost bubble created after applying the corrective
movements and the Morse modifications described above.

We assume first that in the interior of $B_{1}$ there is a Reeb component $R$
and let $T=\partial R$ be the toral leaf of $R.$ Then applying the inverse
Morse modifications which lead $\mathcal{F}^{\prime \prime}$ back to
$\mathcal{F}^{\prime}$ it is clear that at most one of these modifications
affect the Reeb component $R.$ Therefore $R$ either remains a Reeb component
of $\mathcal{F}^{\prime}$ or it is transformed to a truncated Reeb component
of $\mathcal{F}^{\prime}.$ On the other hand, we may see that the inverse of
corrective movements which lead $\mathcal{F}^{\prime}$ back to $\mathcal{F}$
do not affect the Reeb component $R.$ This results from the fact that the
annuli $C$ in a neighborhood of which the corrective movements take place, do
not touch the torus $T.$ Therefore, $\mathcal{F}$ has a Reeb/truncated Reeb component.

Finally, we assume that in the interior of $B_{1}$ there is a truncated Reeb
component $K.$ Then in $B_{1}$ must appear the configuration of Figure
11.\textbf{\ }That is, a solid pair of pants $P$ with two captured spots and a
conic singularity $s$ in $P.$ Therefore, at most one of the inverse Morse
modifications which lead $\mathcal{F}^{\prime \prime}$ back to $\mathcal{F}%
^{\prime},$ may take place in a neighborhood of $s.$ This proves as before
that $K$ remains a truncated Reeb component of $\mathcal{F}$.
\end{proof}

If to the leaves of a Morse/truncated Morse component or to the leaves of a
Reeb/truncated Reeb component we add finitely many trivial bubbles the
components will be called \textit{singular}.

As a corollary of the Theorem above we have.

\begin{corollary}
\label{conclusion}Assuming that $\mathcal{F}$ does not have all its leaves
simply connected then $\mathcal{F}$ has either a singular Morse/truncated
Morse component or a singular Reeb/truncated Reeb component. Therefore, modulo
Morse modifications, $\mathcal{F}$ has either a singular Morse component or a
singular Reeb component.
\end{corollary}

\begin{proof}
If we remove all trivial pair of singularities of $\mathcal{F}$ and we apply
Theorem \ref{corollary2} we obtain the result.
\end{proof}

\section{Stability of Morse foliations}

The goal of this section is to prove the following theorem:

\begin{theorem}
Let $\mathcal{F}$ be a Morse foliation on $S^{3}.$ Then $\mathcal{F}$ is
$C^{1}$-instable unless all the leaves of $\mathcal{F}$ are compact and simply
connected and each leaf contains at most one conic singularity.
\end{theorem}

\begin{proof}
First, we remark that if a foliation $\mathcal{F}$ has a leaf $L$ containing
more than one conic singularity then by a Morse modification we may obtain a
foliation $\mathcal{F}^{\prime}$ arbitrary close to $\mathcal{F}$ in the
$C^{1}$-topology, such that each leaf of $\mathcal{F}^{\prime}$ has at most
one conic singularity. This implies that a Morse foliation $\mathcal{F} $ in
order to be stable should contain at most one conic singularity in the same leaf.

Now we will show that if $\mathcal{F}$ satisfies the assumptions of the
theorem then it is structural stable. Indeed, from Reeb stability theorem
\cite{[Reeb]}, it follows that if $L$ is a non-singular leaf of $\mathcal{F}$
diffeomorphic to $S^{2}$ then $L$ has a tubular neighborhood $V$ homeomorphic
to $S^{2}\times \lbrack-1,1]$ such that the interior of $V$ is foliated by the
product foliation $S^{2}\times \{t\},$ $t\in(-1,1)$ and the leaves $L_{\pm1}$
passing from the points $S^{2}\times \{ \pm1\}$ respectively are singular.
Since $L_{1}$ (resp. $L_{-1})$ is simply connected we have that $L_{1}$ is
homeomorphic to two copies of $S^{2}$ attached to a point. Therefore, Reeb
stability theorem can be applied beyond the leaf $L_{1}$ (resp. $L_{-1}).$
Since $\mathcal{F}$ has finitely many singularities our statement follows.

In the following we assume that $\mathcal{F}$ contains leaves which are
non-compact and simply connected and we will prove that $\mathcal{F}$ is
instable. From Corollary \ref{conclusion}\textbf{\ }we have that if
$\mathcal{F}$ does not have a singular Morse/truncated Morse component then
$\mathcal{F}$ has a singular Reeb/truncated Reeb component. Without loss of
generality, we may assume that all the previous components are not singular
i.e. their leaves do not contain trivial bubbles. Indeed, one may verify that
our proof below is applied for singular components.

So, we assume first that $\mathcal{F}$ has a Morse component $\mathcal{V}.$
Let $T$ be the boundary leaf of $\mathcal{V}$ which is either a torus or a
pseudo-torus. Without loss of generality we may assume that $T$ is a torus and
that there exists a neighborhood $B$ of $T$ such that $\mathcal{F}$ induces on
$B$ a band of leaves i.e. $\mathcal{F}_{|B}$ is the product foliation
$T\times \{t\},$ $t\in \lbrack0,1].$ Indeed, if such a band does not exist then
by a Morse modification we may construct a new foliation $\mathcal{F}^{\prime
}$ which $C^{1}$-close to $\mathcal{F}$ and which contains a band of leaves.
Thus $\mathcal{F}$ is instable. On the other hand, if $\mathcal{F}$ contains a
band of leaves around $T$ then we will show that $\mathcal{F}$ is again
instable. Indeed, it is a standard fact that a stable foliation cannot contain
any band since the product foliation on $B$ can be $C^{1}$-approximated, rel
$\partial(T^{2}\times \lbrack0,1])$ by a foliation with a finite number of
compact leaves.

In the case that $\mathcal{F}$ has a generalized Morse component then the same
proof can be applied and it always results that $\mathcal{F}$ is instable. For
instance, by a Morse modification the generalized Morse component is
transformed to Morse component, say $\mathcal{V},$ and $\mathcal{F}$ is
transformed to a new foliation, say $\mathcal{F}^{\prime}.$ As we have shown,
the existence of $\mathcal{V}$ implies that $\mathcal{F}^{\prime}$ is instable
and hence the initial foliation $\mathcal{F}$ is instable too. Therefore, we
may assume in the following that $\mathcal{F}$ does not have neither Morse
components nor truncated Morse components

Now, without loss of generality, we may assume that $\mathcal{F}$ has only
Reeb components. Indeed, by applying finitely many Morse modifications,
$\mathcal{F}$ is transformed to new Morse foliation, say $\mathcal{F}^{\prime
},$ which has only Reeb components. Obviously, if we show that $\mathcal{F}%
^{\prime}$ is instable then the same will be true for $\mathcal{F}$.

\textit{Claim 1}: One of the toral leaves, say $T,$ of some Reeb component of
$\mathcal{F}^{\prime}$ should be flat, i.e. all the elements of the holonomy
group of $T$ have $\infty$ contact with the identity \cite{[Ros-Rou]}.

The proof of this claim is similar to the proof of the corresponding statement
(see \cite{[Ros-Rou]}, Proof of Theorem 1.1). We have to notice only that the
method of erasing Reeb components, as it is explained in \cite{[Ros-Rou]}, if
it is applied for Morse foliations does not create neither Morse components
nor foliations with all the leaves compact and simply connected. This proves
that in our case, all Reeb components of $\mathcal{F}^{\prime}$ cannot be
erased so some toral leaf $T$ is flat.

\textit{Claim 2}: The leaf $T$ can be thickened, i.e. $\mathcal{F}^{\prime}$
can be $C^{1}$-approximated by a foliation $\mathcal{F}^{\prime \prime}$ such
that $\mathcal{F}^{\prime \prime}=\mathcal{F}^{\prime}$ outside a tubular
neighborhood $V$ of $T,$ and $\mathcal{F}^{\prime \prime}$ has a band of toral leaves.

The proof of this claim is identical to the proof of Lemma a in
\cite{[Ros-Rou]}.

As we noted above, the band $T\times \lbrack0,1]$ of toral leaves can be
$C^{1}$-approximated, rel $\partial(T\times \lbrack0,1])$ by a foliation with a
finite number of compact leaves. Therefore $\mathcal{F}^{\prime}$ can be
$C^{1}$-approximated by a Morse foliation $\mathcal{G}^{\prime}$ for which
there does not exist a homeomorphism $h$ sending $\mathcal{G}^{\prime}$ to
$\mathcal{F}^{\prime}.$ Now, going back to $\mathcal{F}$ by applying to
$\mathcal{F}^{\prime}$ the inverse Morse modifications, it is easy to see that
by means of $\mathcal{G}^{\prime},$ we may construct a foliation $\mathcal{G}$
which is $C^{1}$-close to $\mathcal{F}$ and which is not topological
equivalent to $\mathcal{F}$.
\end{proof}

\begin{remark}
Theorem \ref{stability} can be easily derived from the previous theorem since
Morse foliations are generic in the space of Haefliger structures.
\end{remark}

\bigskip \noindent Charalampos Charitos\newline \noindent Department of Natural
Resources Management \& Agricultural Engineering\newline \noindent Agricultural
University Athens\newline \noindent75, Iera Odos, Athens - GREECE\newline%
\noindent e-mail: bakis@aua.gr

\end{document}